\documentclass[11pt,twoside]{article}
\usepackage{fullpage}

\usepackage{epsf}
\usepackage{fancyhdr}
\usepackage{graphics}
\usepackage{graphicx}
\usepackage{psfrag}
\usepackage{microtype}
\usepackage{subfigure}
\usepackage{algorithmic}
\usepackage{color,xcolor}
\usepackage[linesnumbered,ruled]{algorithm2e}
\DontPrintSemicolon
\usepackage{color}

\usepackage{textcmds}
\usepackage{mathrsfs}
\usepackage{amsthm}
\usepackage{amsfonts}
\usepackage{amsmath}
\usepackage{amssymb,bbm}
\usepackage{bigints}

\newcommand{\Ecal}{\ensuremath{\mathcal{E}}}

\DeclareMathOperator{\sign}{sign}

\newtheoremstyle{named}{}{}{\itshape}{}{\bfseries}{.}{.5em}{\thmnote{#3's }#1}
\theoremstyle{named}

\theoremstyle{plain}

\newtheorem{theorem}{Theorem}[section]
\newtheorem{proposition}{Proposition}[section]
\newtheorem{lemma}[proposition]{Lemma}

\newtheorem{corollary}[theorem]{Corollary}
\newtheorem{definition}[proposition]{Definition}

\newlength{\widebarargwidth}
\newlength{\widebarargheight}
\newlength{\widebarargdepth}

\long\def\comment#1{}
\definecolor{carnelian}{rgb}{0.7, 0.11, 0.11}
\definecolor{battleshipgrey}{rgb}{0.52, 0.52, 0.51}
\definecolor{darkgray}{rgb}{0.66, 0.66, 0.66}
\definecolor{indiagreen}{rgb}{0.07, 0.53, 0.03}
\definecolor{darkgreen}{rgb}{0.0, 0.2, 0.13}
\definecolor{darkspringgreen}{rgb}{0.09, 0.45, 0.27}
\definecolor{dukeblue}{rgb}{0.0, 0.0, 0.61}
\definecolor{olivedrab7}{rgb}{0.24, 0.2, 0.12}
\definecolor{darkblue}{rgb}{0.0, 0.0, 0.55}
\definecolor{darkscarlet}{rgb}{0.34, 0.01, 0.1}
\definecolor{candyapplered}{rgb}{1.0, 0.03, 0.0}
\definecolor{ao(english)}{rgb}{0.0, 0.5, 0.0}
\definecolor{applegreen}{rgb}{0.55, 0.71, 0.0}

\makeatletter
\long\def\@makecaption#1#2{
        \vskip 0.8ex
        \setbox\@tempboxa\hbox{\small {\bf #1:} #2}
        \parindent 1.5em  
        \dimen0=\hsize
        \advance\dimen0 by -3em
        \ifdim \wd\@tempboxa >\dimen0
                \hbox to \hsize{
                        \parindent 0em
                        \hfil
                        \parbox{\dimen0}{\def\baselinestretch{0.96}\small
                                {\bf #1.} #2
                                }
                        \hfil}
        \else \hbox to \hsize{\hfil \box\@tempboxa \hfil}
        \fi
        }
\makeatother

\RequirePackage{natbib}
\usepackage{url}
\usepackage[colorlinks,linkcolor=magenta,citecolor=blue, pagebackref=true,backref=true]{hyperref}
\renewcommand*{\backrefalt}[4]{%
    \ifcase #1 \footnotesize{(Not cited.)}%
    \or        \footnotesize{(Cited on page~#2.)}%
    \else      \footnotesize{(Cited on pages~#2.)}%
    \fi}

\long\def\comment#1{}
\renewcommand\vec[1]{\ensuremath\boldsymbol{#1}}

\newcommand{\Ocal}{\ensuremath{\mathcal{O}}}
\newcommand{\Gcal}{\ensuremath{\mathcal{G}}}

\newcommand{\excessmass}{\ensuremath{\mathcal{EX}}}

\makeatletter
\newcommand\mathcircled[1]{%
  \mathpalette\@mathcircled{#1}%
}
\newcommand\@mathcircled[2]{%
  \tikz[baseline=(math.base)] \node[draw,circle,inner sep=1pt] (math) {$\m@th#1#2$};%
}
\makeatother

\theoremstyle{plain}

\numberwithin{remark}{section}

\renewenvironment{abstract}
 {\small
  \begin{center}
  \bfseries \abstractname\vspace{-.5em}\vspace{0pt}
  \end{center}
  \list{}{%
    \setlength{\leftmargin}{15mm}
    \setlength{\rightmargin}{\leftmargin}%
  }%
  \item\relax}
 {\endlist}

\begin{document}
\begin{center}

{\bf{\LARGE{On Excess Mass Behavior in Gaussian Mixture Models with Orlicz-Wasserstein Distances}}}

\vspace*{.2in}
 {\large{
 \begin{tabular}{ccc}
  Aritra Guha$^{\star}$ & Nhat Ho$^{\dagger}$ &  XuanLong Nguyen$^{\star}$
 \end{tabular}
}}

 \vspace*{.2in}

 \begin{tabular}{c}
 Department of Statistics, University of Michigan$^\star$\\
 Department of Statistics and Data Sciences,
The University of Texas at Austin$^\dagger$
 \end{tabular}

\vspace*{.2in}

\today

\vspace*{.2in}

\begin{abstract}
Dirichlet Process mixture models (DPMM) in combination with Gaussian kernels have been an important modeling tool for numerous data domains arising from biological, physical, and social sciences. However, this versatility in applications does not extend to strong theoretical guarantees for the underlying parameter estimates, for which only a logarithmic rate is achieved. In this work, we (re)introduce and investigate a metric, named Orlicz-Wasserstein distance, in the study of the Bayesian contraction behavior for the parameters. We show that despite the overall slow convergence guarantees for all the parameters, posterior contraction for parameters happens at almost polynomial  rates in outlier regions of the parameter space. Our theoretical results provide new insight in understanding the convergence behavior of parameters arising from various settings of hierarchical Bayesian nonparametric models. In addition, we provide an algorithm to compute the metric by leveraging Sinkhorn divergences and validate our findings through a simulation study.

\end{abstract}
\end{center}
\section{Introduction}
\label{sec:introduction}
From their origin in the work of Pearson~\cite{Pearson-1894}, mixture models have been widely used by statisticians~\cite{McLachlan-Basford-88,Lindsay-95,Mengersen-etal-2001} in variety of modern interdisciplinary domains such as medical science~\cite{Schlattmann-09}, bioinformatics~\cite{bioinfo-appli-mm}, survival analysis~\cite{Survival-MM}, psychometry~\cite{Yuqi-psychometrika} and image classification~\cite{GMM-Image}, to name just a few. The heterogeneity in data populations and associated quantities of interest has inspired the use of a variety of kernels, each with its own advantages and characteristics. Gaussian kernels are particularly popular in various inferential problems, especially those related to density estimation and clustering analysis~\cite{Kotz-etal,Bailey-94, Wasserman-97,Robert-96, Banfield-Raftery-93}. In addition to the choice of kernels, the Bayesian mixture modelers are also guided by the selection of prior distributions for the quantities of interest. In particular, Bayesian nonparametric priors (BNP) for mixture models are increasingly embraced, thanks to computational ease and the modeling flexibility that these rich priors entail~\cite{Escobar-West-95,MacEachern-99}.   

On the theoretical front, convergence rates for (Gaussian) mixture models received extensive treatments in the Bayesian paradigm~\cite{Ghosal-Ghosh-vanderVaart-00, Barron-Shervish-Wasserman, Ghosal-vanderVaart-07}. There have been enormous recent progress on both density estimation and parameter estimation problems. The density estimation problem under Gaussian mixture models with BNP priors was extensively studied by~\cite{Ghosal-vanderVaart-01} who obtained attractive polynomial rates of contraction relative to the Hellinger distance metric. In the parameter estimation problem, the metric of choice is Wasserstein distance, which proved to be a natural tool to analyze the convergence of mixture parameters \cite{Nguyen-13}. Moreover,~\cite{Nguyen-13} showed that the fast rates for density estimation with BNP Gaussian mixtures do not extend themselves to parameter estimation scenarios. Meanwhile, practitioners have employed successfully BNP mixture models, which yield useful estimates for model parameters that provide meaningful information about the data population's heterogeneity. This state of affairs leaves a gap in the theoretical understanding and the practical usage of Bayesian mixture models. In this paper, we aim to bridge this gap by capturing more accurately the heterogenenous behavior in the rates of parameter estimation. We proceed to describe this in further detail.



\subsection{Gaussian Mixture Models} 
Consider discrete \emph{mixing (probability) measure} $G=\sum_{i=1}^{k} p_i \delta_{\theta_i}$. Here, $\vec{p}=(p_1,\dots,p_k)$ is a vector of mixing weights, while atoms $\{\theta_i\}_{i=1}^{k}$ are elements in a given space $\Theta \subset \mathbb{R}^d$. Here $k$ is used to denote the number of components, which can potentially be infinite. Mixing measure $G$ is combined with a Gaussian kernel with known covariance matrix $\Sigma$, denoted by $f(\cdot|\theta) \sim \mathcal{N}(\theta, \Sigma)$ (to avoid notational cluttering, we remove $\Sigma$ from notation in the remainder of the paper), with respect to the Lebesgue measure $\mu$ to yield a mixture density: 
\begin{eqnarray} 
\label{eq:mixture}
p_{G}(.) :=  \int f(\cdot|\theta)\mathrm{d}G(\theta)=\sum_{i=1}^{k} p_i f(\cdot|\theta_i).
\end{eqnarray}
The atoms $\theta_i$'s are representatives of the underlying subpopulations.  Let $X_1,\ldots, X_n$ be i.i.d. samples from a mixture density $p_{G_0}(x) = \int f(x|\theta) \textrm{d}G_0(\theta)$, where $G_0=\sum_{i=1}^{k_0} p_i^0 \delta_{\theta_i^0}$ is a true but unknown discrete mixing measure with \emph{unknown} number of support points $k_0 \in \mathbb{N} \cup \{\infty\}$. We assume in this work that all the masses $\{p_i^0\}_{i=1}^{k_0}$ are strictly positive and the atoms $\{\theta_i^0: i\leq k_0\}$ are distinct. 


 A Bayesian mixture modeler places a prior distribution $\Pi_n$ on a suitable space (specifically, $\overline{\Gcal}(\Theta)$ of discrete measures on $\Theta$). The posterior distribution corresponding to $\Pi_n$, both of which may vary with sample size, can be computed as:
\begin{eqnarray}
\label{eq:Bayes}
\Pi_n(G \in B \bigr|X_{1:n}) =\frac{\int_B \prod_{i=1}^n p_G(X_i) \mathrm{d}\Pi_n(G)}{\int_{\overline{\Gcal}(\Theta)} \prod_{i=1}^n p_G(X_i) \mathrm{d}\Pi_n(G)}.
\end{eqnarray} 


\vspace{0.5 em}
\noindent
\textbf{Dirichlet process Gaussian mixture models:} In the absense of the knowledge of the number of mixture components $k_{0}$, the learning of mixture models is carried out by the use of Bayesian non-parametric (BNP) priors, leading to the \emph{infinite mixture} setting. One of the most popular such priors is the Dirichlet process prior \cite{Antoniak-74}, which uses sample draws from a base measure $H$ to define the random components and weights of the mixture model, leading to the popular \emph{Dirichlet Process Gaussian Mixture Models} (DPGMM)~\cite{Lo-84,Escobar-West-95}. In essence, the Dirichlet process prior places zero probability on mixing measures with a finite number of supporting atoms and enables the addition of more atoms in the supporting set as the number of data points increase. The DPGMM is formulated as follows:
\begin{eqnarray}
\label{eq:DPMM}
G & \sim & \text{DP}(\alpha,H), \nonumber \\
\theta_{1},\ldots,\theta_{n} & \overset{i.i.d.}\sim & G, \nonumber \\
X_{i}|\theta_{i} & \sim & f(X_{i}|\theta_{i}), \quad \forall i=1,\ldots,n, 
\end{eqnarray}
where \text{DP} stands for Dirichlet process, the base measure $H$ is a distribution on $\Theta$, and $\alpha > 0$ is a concentration parameter which controls the rate at which new atoms may be considered, by varying the tail-behavior of mixture weights.  A parametric counterpart of DPGMM is the mixture of finite Gaussian mixtures prior (MFM)~\cite{Miller-2016}, which places all its mass on mixing measures with finite number of supporting atoms.
BNP priors other than DPGMM may have the effects of pushing the atoms away from each other~\cite{Xu-Xie-Repulsive-17} or encouraging the weights of mixture to have a polynomial tail behavior~\cite{deblasi_gibbs}.

The popularity of BNP priors may partially have been promoted due to a misconception that it "automatically" determines the number of components in the posterior inference process. This issue was highlighted by~\cite{Miller-2014}, who demonstrated that Dirichlet Process priors overestimate the true number of components, $k_{0}$, almost surely. 
Subsequent work~\cite{Guha-MTM-19} has provided post-processing techniques to determine $k_{0}$ consistently with Dirichlet Process priors. Their method depends on the knowledge of the parameter contraction rate, with respect to the \emph{Euclidean Wasserstein metric}, i.e., Wasserstein metric with underlying distance metric $\ell_2$, a rate that is extremely slow for the Gaussian kernels \cite{Nguyen-13}.

The inconsistency of estimating $k_{0}$ arises primarily  because Dirichlet priors typically tend to create a large number of extraneous components. While some of these components may be in the neighborhood of the true supports, others may be outliers and in practice, can be easily eliminated from consideration by careful truncation techniques. However, the Euclidean Wasserstein distance treats both the scenarios similarly and in turn yields slow convergence rates for both sets of extraneous atoms. This calls for alternative metrics for investigating parameter estimation rates. In a recent work,~\cite{manole-nhat-icml} argued that Wasserstein metrics capture only the worst-case uniform rates of parameter estimation and therefore can yield extremely slow rates in comparison to the local rates observed in practice, which may vary drastically based on the likelihood curvature in the parameter neighborhood. Employing alternate distance metrics via the use of Voronoi tessallations, they showed that in the finite Gaussian mixture setting with overfitted components (where $\infty>k>k_0$), even though the uniform convergence rates may be slow as $k$ increases, there may still be some atoms which enjoy much faster rates of convergence. 

 The infinite Gaussian mixture setting is generally more challenging to address, (a) since the "true" atoms are not guaranteed to be well-separated, (b) each true atom may be surrounded by potentially infinitely many atoms a posteriori and (c) a posteriori samples can potentially have a significant portion of atomic masses attributed to outlier regions of the parameter space. We argue in this work that in the infinite Gaussian mixture setting, the rates captured by Wasserstein distances for outlier masses are inadequately slow and will demonstrate that with the help of a new suitably defined choice of metric this difficulty can be alleviated.

\subsection{Contribution}
As a primary contribution of this work we study a generalized class of metrics called \emph{Orlicz-Wasserstein} metrics, in the context of parameter estimation arising in infinite mixture models. We show that an in-depth analysis using this metric helps alleviate a number of the concerns attributable to the use of Wasserstein distances for quantifying the rates of parameter convergence arising in infinite Gaussian mixtures. This class of distance metrics generalizes the Wasserstein metric relative to the Orlicz norm using a variety of choices of convex functions. They encompass a very wide range of distances on the space of probability measures, including the Euclidean Wasserstein metrics as a special case. By making appropriate choices of convex functions we can obtain a fast, almost polynomial contraction rates for atomic masses in outlier regions of the parameter space. This is very different from the slow local contraction behavior around the true atoms under the standard Wasserstein metric. This helps us establish informative and useful finer details about the convergence behavior of parameter estimates underlying the usage of Gaussian mixture models in clustering. We believe the usage of Orlicz-Wasserstein metrics for parameter estimation in Dirichlet process Gaussian mixture models opens a new range of directions for future research that aim for developing statistically sound and computationally efficient strategies for posterior sampling with mixture models.

\vspace{0.5 em}
\noindent
\textbf{Organization.} The remainder of the paper is organized as follows. Section \ref{Section:preliminary} provides necessary backgrounds about posterior contraction of parameters in Gaussian mixture models under Wasserstein distances. Section~\ref{sec:orlicz-Wasserstein_distance} introduces \emph{Orlicz-Wasserstein} distances and some of its key properties. Section~\ref{ssection:computation} provides computational approximations to calculating Orlicz-Wasserstein metrics for two mixing measures. Section~\ref{Sec:lower_bound_hellinger} presents exact lower bounds for the Hellinger metric with respect to Orlicz-Wasserstein distances for Gaussian kernels. Section~\ref{ssection:posterior_contraction_gaussian} uses the results in Section~\ref{Sec:lower_bound_hellinger} to provide the key results in the paper with regards to contraction behavior using Orlicz-Wasserstein metrics. Proofs of results are deferred to the Appendices.

\vspace{0.5 em}
\noindent
\textbf{Notation.} For any function $g:\mathcal{X} \to \mathbb{R}$, we denote $\widetilde{g}(\omega)$ as the Fourier transformation of function $g$. Given two densities $p, q$ (with respect to the Lebesgue measure $\mu$), the squared Hellinger distance is given by  $h^{2}(p,q)= {\displaystyle (1/2) \int (\sqrt{p(x)}-\sqrt{q(x)})^{2}\textrm{d}\mu(x)}$. 
For any metric $d$ on $\Theta$, we define the open ball of $d$-radius $\epsilon$ around $\theta_0 \in \Theta$ as $B_d(\epsilon,\theta_0)$. Additionally, the expression $a_{n} \gtrsim b_{n}$ will be used to denote the inequality up to a constant multiple where the value of the constant is independent of $n$. We also denote $a_{n} \asymp b_{n}$ if both $a_{n} \gtrsim b_{n}$ and $a_{n} \lesssim b_{n}$ hold. Furthermore, we denote $A^{c}$ as the complement of set $A$ for any set $A$ while $B(x,r)$ denotes the ball, with respect to the $l_2$ norm, of radius $r > 0$ centered at $x \in \mathbb{R}^{d}$. The expression $D(\epsilon,\mathscr{P},d)$ used in the paper denotes the $\epsilon$-packing number of the space $\mathscr{P}$ relative to the metric $d$. $d$ is replaced by $h$ to denote the hellinger norm. Finally, we use $\text{Diam}(\Theta)= \sup \{\|\theta_1-\theta_2\|: \theta_1,\theta_2 \in \Theta\}$  to denote the diameter of a given parameter space $\Theta$ relative to the $l_{2}$ norm, $\|\cdot\|$, for elements in $\mathbb{R}^{d}$. Regarding the space of mixing measures, let $\Ecal_{k}:=\Ecal_{k}(\Theta)$ and $\Ocal_{k}:=\Ocal_{k}(\Theta)$ respectively denote the space of all mixing measures with exactly and at most $k$ support points, all in $\Theta$. Additionally, denote $\Gcal : = \Gcal(\Theta) = \mathop{\cup} \limits_{k \in \mathbb{N}_{+}}{\Ecal_{k}}$ the set of all discrete measures with finite supports on $\Theta$. $\overline{\Gcal}(\Theta)$ denotes the space of all discrete measures (including those with countably infinite supports) on $\Theta$. Finally, $\mathcal{M}(\Theta)$ stands for the space of all probability measures on $\Theta$.

\section{Posterior contraction under Wasserstein distance} 
\label{Section:preliminary}
Following the work of~\cite{Nguyen-13}, Wasserstein distances have been used to explore parameter estimation rates of mixture models, embodied through their mixing measures. In this section, we outline the basic concepts as follows. Let $\Theta \subset \mathbb{R}^d$. Moreover, define $\mathcal{M}(\Theta)=\{P:P \text{ is a probability measure on } \Theta \}$.
\begin{definition}
\label{def:Wasserstein}
Given $\mu, \nu \in \mathcal{M}(\Theta)$ and the $l_2$ metric $\|\cdot\|$ on $\mathbb{R}^d$,  the Wasserstein distance~\cite{Villani-09} of order $r$ seeks a joint measure $\pi \in \Pi$ minimizing 
\begin{eqnarray}
\label{eq:Kantorovich}
W_r(\mu, \nu) := \left(\inf_{\pi \in \Pi} \int_{ \Theta  \times \Theta} \|\theta_1- \theta_2\|^r d \pi(\theta_1, \theta_2) \right)^{1/r}.
\end{eqnarray}
Here, $\Pi$ is the set of couplings of $\mu$ and $\nu$ denoted by $\Pi=\{\pi: \gamma_\#^1\pi = \mu, \gamma_\#^2 \pi = \nu\}$, where $\gamma^1$, $\gamma^2$ are functions that project onto the first and second coordinates of $\Theta \times \Theta$ respectively. 
\end{definition}
In particular, as shown by~\cite{Nguyen-13}, given two discrete measures $G=\mathop {\sum }_{i=1}^{k}{p_{i}\delta_{\theta_{i}}}$ 
and
$G' = \sum_{i=1}^{k'}p'_i \delta_{\theta'_i}$, a coupling between $\vec{p}$ and
$\vec{p'}$ is a joint distribution $\vec{q}$ on $[1\ldots,k]\times [1,\ldots, k']$, which
is expressed as a matrix
$\vec{q}=(q_{ij})_{1 \leq i \leq k,1\ \leq j \leq k'} \in [0,1]^{k \times k'}$
with marginal probabilities
$\mathop {\sum }_{i=1}^{k}{q_{ij}}=p_{j}'$ and $\mathop {\sum  }{j=1}^{k'}{q_{ij}}=p_{i}$ for any $i=1,2,\ldots,k$ and $j=1,2,\ldots,k'$.
We use $\mathcal{Q}(\vec{p},\vec{p'})$ to denote the space of all such couplings of $\vec{p}$ and $\vec{p'}$. For any $r \geq 1$, the $r$-th order Wasserstein distance between
$G$ and $G'$ is given by
\begin{eqnarray}
W_{r}(G,G') & = & \inf_{\vec{q} \in \mathcal{Q}(\vec{p},\vec{p'})}\biggr ({\mathop {\sum }\limits_{i,j}{q_{ij}\|\theta_{i}-\theta_{j}'\|^{r}}}\biggr )^{1/r}.
\end{eqnarray}
~\cite{heinrich-kahn} show that with Gaussian kernels, the minimax rate for estimation is dependent on the number of extra components and goes down as the number of potential components increases, meaning it gets harder to accurately cluster the observations as we have more and more extra components. 
The Gaussian kernel being smooth fits in as many components as possible without changing the mixture density and therefore achieves a very slow parameter contraction rate.  With potentially infinitely many extra components (while using Dirichlet Process priors), rates are even slower. In fact,~\cite{Nguyen-13} shows that for DPGMM with posterior distribution $\Pi_n(\cdot|X_{1:n})$, the following holds true.

{\small
\begin{eqnarray}
\label{eq:infinite_Gaussian_Nguyen}
\Pi_n \biggr(G \in \overline{\mathcal{G}}(\Theta) : W_2(G,G_0) 
		 \lesssim (\log n)^{-1/2} \biggr| 
		X_{1:n}\biggr) \to 1 
\end{eqnarray}
}
in $p_{G_0}$-probability. On the other hand, it has been shown that ordinary-smooth kernels need only a power of $-\log(\epsilon)$ components to approximate an infinite component mixing density upto $\epsilon$- approximation in $\mathbb{L}_q$ distance \cite{Nguyen-13, Fengnan_2016}. Correspondingly, Laplace kernels need a polynomial power of $(1/\epsilon)$ many components for the same degree of approximation. This combined with~\eqref{eq:infinite_Gaussian_Nguyen} suggests that BNP priors use a lot more extra components to fit the true mixture distribution than is necessary, especially with Gaussian kernel. The extra components can potentially arise from two different sources, (i) multiple supporting atoms in the posterior trying to approximate each true atom, (ii) or excessively many outlier atoms in the posterior sample. If condition (ii) is true, this may  potentially have negative consequences for using Gaussian kernels for clustering purposes. From Eq.~\eqref{eq:infinite_Gaussian_Nguyen}, we are only able to conclude that 
\begin{eqnarray}
\label{eq:excess_mass_wasserstein}
    \Pi_n \biggr(G = \sum p_i \delta_{\theta_i} : \sum_j p_j \mathbbm{1}_{\{ \|\theta_j-\theta^0_i\|> \eta \ \forall i \}}   \gtrsim & \log(n)/\eta^2 \biggr|X_{1:n} \biggr) \to 1 \
\end{eqnarray}
which states that masses attributed to outlier atoms (those $>\eta$ distance from any "true" atom) vanish at only a slow logarithmic rate. Clearly, while standard Wasserstein distances are the popular choices of metrics, they do not help differentiate between the sources of extra atoms, and thereby are not useful while discarding outlier atoms. To facilitate this distinction of the source of excess atoms, in this paper we consider a generalisation of standard (Euclidean) Wasserstein metrics called \emph{Orlicz-Wasserstein} distances which allow placement of higher weight penalties on outliers and thereby help to identify outlier atoms better. We proceed in the following sections to describe this in further detail.



\section{A generalized metric for contraction of mixing measures}
\label{section:gaussian_excess_mass}
In existing literature thus far, the rates of parameter estimation have been extensively studied with respect to Euclidean Wasserstein distances, in the works of~\cite{Nguyen-13, Ho-Nguyen-EJS-16, Ho-Nguyen-AOS-17, Fengnan_2016, Guha-MTM-19}. As part of this work, we extend such results to the regime of \emph{Orlicz-Wassertein} metrics which take a more careful consideration of the geometry of the parameter space. In that regard, for the sake of completeness, we first introduce the reader to the notion of Orlicz norms and spaces as follows.

\subsection{Orlicz-Wasserstein distance}
\label{sec:orlicz-Wasserstein_distance}
The Orlicz norm is defined as follows~\cite{wellner-orlicz}.
\begin{definition}
\label{definition: orlicz function}
Let $\mu$ be a $\sigma-$finite measure on a space $\mathcal{X}$ with metric $\|\cdot\|$. Assume that $\Phi :[0,\infty) \to [0,\infty)$ be a convex function satisfying:
\begin{enumerate}
    \item[(i)] $\dfrac {\Phi (x)}{x}\to \infty ,\quad {\text{as }}x\to \infty $,
    \item[(ii)] $\dfrac {\Phi (x)}{x}\to 0,\quad {\text{as }}x\to 0$.
\end{enumerate}
Then, the Orlicz space is defined as follows:
\begin{eqnarray}
  L_{\Phi} :=\big\{f: \mathcal{X} \rightarrow \mathbb{R} | \ \exists \ \lambda \in \mathbb{R}^{+} \text{ s.t.} \int _{X}\Phi (\|f(x)\|/\lambda)\,d\mu(x) \leq 1\big\}.
\end{eqnarray}
Moreover, the Orlicz norm corresponding to $ f \in L_{\Phi}$ is given by:
\begin{eqnarray}
     \|f\|_{\Phi}:= \inf\{\lambda \in \mathbb{R}^{+}: \int_{X}\Phi (\|f(x)\|/\lambda)\,d\mu(x) \leq 1\}.
\end{eqnarray}
\end{definition}

Without loss of generalisation, we will assume $\mathcal{X} = \mathbb{R}^d$, with $\|\cdot\|$ denoting the standard Euclidean metric. Notice that when $\Phi(x)=x^p$ with $p \geq 1$, the Orlicz norm, $\|f\|_{\Phi}$ is the same as the $\mathbb{L}_p$-norm. In this sense, the Orlicz norm generalizes the concept of $\mathbb{L}_p$-norm for $p \geq 1$. Recall that, a coupling between two probability measures $\nu_1$ and $\nu_2$ on $\mathbb{R}^d$ is a joint distribution on $\mathbb{R}^d \times \mathbb{R}^d$ with corresponding marginal distributions $\nu_1$ and $\nu_2$. Corresponding to the Orlicz norms, we define the \emph{Orlicz-Wasserstein metric} which generalizes the $W_r$-metric as follows. 
\begin{definition}
\label{def:Orlic-Wasserstein}
Let $\nu_1,\nu_2$ be probability measures on $(\mathbb{R}^d,\|\cdot\|)$. Assume that $\Phi:[0,\infty) \to [0,\infty)$ is a convex function satisfying conditions (i) and (ii) in Definition~\ref{definition: orlicz function}. We define the \emph{Orlicz-Wasserstein} distance between $\nu_{1}$ and $\nu_{2}$ as follows:
\begin{eqnarray}
\label{eq:orlicz-Wasserstein}
     W_{\Phi}(\nu_1,\nu_2):= \inf_{\nu \in \mathcal{Q}(\nu_1,\nu_2)} \inf \{\lambda \in \mathbb{R}^{+}:  
     \int _{\mathbb{R}^d \times \mathbb{R}^d}\Phi (\|x-y\|/\lambda)\,d\nu(x,y) \leq 1\},
\end{eqnarray}
where $\mathcal{Q}(\nu_1,\nu_2)$ is the set of all possible couplings of $\nu_1$ and $\nu_2$.
\end{definition}

Orlicz Wasserstein distances have been briefly introduced in the works of~\cite{Kell-orlicz-wasserstein,Sturm-orlicz}, however, the utility of the metrics for contraction properties of parameter estimation has remained hitherto unexplored. Also, following Lemma 3.1 of~\cite{Sturm-orlicz}, we see under some minor regularity conditions, for every $\Phi,\nu_1,\nu_2$, there exists $\lambda_{\min}$ and $\nu_{\text{opt}}$ such that $\lambda_{\min}= W_{\Phi}(\nu_1,\nu_2)$ and $\int _{\mathbb{R}^d \times \mathbb{R}^d}\Phi (\|x-y\|/\lambda_{\min})\,d\nu_{\text{opt}}(x,y) = 1$. This combined with Fubini's theorem establishes the equivalence of the definitions in this work and those of~\cite{Sturm-orlicz,Kell-orlicz-wasserstein}.

\noindent
Note that when $\Phi(x)=x^r$ for $r \geq 1$, then $W_{\Phi}(\nu_1,\nu_2)=W_r(\nu_1,\nu_2)$, the usual Wasserstein distance of order $r$ between $\nu_{1}$ and $\nu_{2}$. The following lemma demonstrates that Orlicz-Wasserstein defines a proper metric on $(\mathbb{R}^d,\|\cdot\|)$.
\begin{lemma}
\label{lemma: orlicz_metric}
The Orlicz-Wasserstein $W_{\Phi}$ is a distance metric on the set of probability measures on $(\mathbb{R}^d,\|\cdot\|)$, namely, it is symmetric and satisfies the identity and triangle inequality properties.
\end{lemma}
\noindent
The proof of Lemma~\ref{lemma: orlicz_metric} is in Appendix~\ref{subsec:proof:lemma: orlicz_metric}. The notion of Orlicz-Wasserstein distance may encompass a  stronger notion of metrics than that of the usual Wasserstein distance to compare probability measures as evidenced by the following lemma.

\begin{lemma}
\label{lemma: orlicz properties}
Let $\nu_1,\nu_2$ be probability measures on $(\mathbb{R}^d,\|\cdot\|)$. Also assume $\Phi,\Psi$ are convex functions satisfying conditions (i) and (ii) in Definition~\ref{definition: orlicz function}. Suppose that for all $x>0$, $\Phi(x) \leq \Psi(x)$. Then, we have 
\begin{eqnarray}
W_{\Phi}(\nu_1,\nu_2) \leq W_{\Psi}(\nu_1,\nu_2) \nonumber.
\end{eqnarray}
\end{lemma}

The proof of Lemma~\ref{lemma: orlicz properties} is in Appendix~\ref{subsec:proof:lemma: orlicz properties}. Note that the supremum of convex functions is also a convex function. Therefore, as a corollary to the above lemma we obtain the following inequality.
\begin{corollary}
Let $\Phi_1(\cdot)$ be a polynomial convex function and $\Phi_2(\cdot)$ an exponential convex function. $\Psi$ is the supremum of $\Phi_1(\cdot)$ and $\Phi_2(\cdot)$. Then the following holds, for any $G,G'$, $1>\alpha>0$.
\begin{eqnarray}
    W_{\Psi}(G,G')  &\geq& W_{\alpha\Phi_1 +(1-\alpha) \Phi_2}(G,G')  \\ &\geq& \alpha W_{\Phi_1}(G,G') +(1-\alpha) W_{\Phi_2}(G,G') \nonumber
\end{eqnarray}
\end{corollary}
An important property of the Wasserstein distances is that if one mixing measure is close to another in Wasserstein distance, it provides a way to control the corresponding contraction rates of the atoms and the masses associated with them. The following lemma provides a similar result for Orlicz-Wasserstein norms.
\begin{lemma}
\label{lemma: orlicz_bound}
Let $G_0=\sum_{i=1}^{k_0} p^0_i\delta_{\theta^0_i}$, $G=\sum_{i=1}^k p_i\delta_{\theta_i}$ be mixing measures such that $\theta_j,\theta^0_i \in \mathbb{R}^d$ for all $i,j$. Assume that $\Phi:[0,\infty) \to [0,\infty)$ is a convex function satisfying conditions (i) and (ii) in Definition~\ref{definition: orlicz function}. Then 
\begin{eqnarray}
 \sum_j p_j\mathbbm{1}_{\{\|\theta_j-\theta^0_i\|>\eta \text{ for all } i\}}\leq \left( \Phi\left(\dfrac{\eta}{W_{\Phi}(G,G_0)}\right)\right)^{-1}.
\end{eqnarray}
Here, $k_0,k$ can also take the value $\infty$.
\end{lemma}
\noindent
The proof of Lemma~\ref{lemma: orlicz_bound} is in Appendix~\ref{subsec:proof:lemma: orlicz_bound}. Lemma~\ref{lemma: orlicz_bound} allows us to identify the amount of mass transferred over large distances, when the mass transfer occurs between two measures $G$ and $G_0$. Note that the constraint on $\Phi$ is very minimal, thereby lending flexibility to the result. Since operations like supremums of convex functions or compositions of a convex function with a non-decreasing convex function (this is the outer function), also yield convex functions, Lemma~\ref{lemma: orlicz_bound} is a standalone result of interest as a generalisation of Bernstein/Hoeffding type inequalities for mixing measures.
\subsection{Lower bound of Hellinger distance based on Orlicz-Wasserstein metric}
\label{Sec:lower_bound_hellinger}
In the previous section, we state results to control the cost of mass transfer attributable to large transportation distances using Orlicz-Wasserstein distances. This is an important result in understanding the contraction behaviors of support points in the outlier regions of parameter space. Traditionally, contraction behavior has been extensively studied~\cite{Ghosal-vanderVaart-01} in the regime of mixture densities $p_G$. The following results help us connect our understanding of posterior contraction on space of mixture densities to that of mixing measure, relative to that of Orlicz-Wasserstein distances. This is stated as follows in the next theorem.
\begin{theorem}
\label{theorem: gaussian_excess}
 Let $\Phi$ be a convex function satisfying conditions (i) and (ii) in Definition~\ref{definition: orlicz function} such that $\Phi(x) \leq \exp(x^{\beta})-1$ for some $16/15 >\beta>1$. Then, as $\Theta = [-\bar{\theta},\bar{\theta}]^{d}$, for any mixing measures $G,G'$, we have
\begin{eqnarray}
\label{eq: orlicz-Wasserstein-hellinger}
 W_{\Phi}(G,G') \lesssim  C \biggr( \dfrac{ \bar{\theta}^{5/4}}{(\log(1/h(p_{G},p_{G'})))^{1/8}}  + \left(\dfrac{1}{\log(1/h(p_{G},p_{G'}) )}\right)^{11/8} \\  +    \left(\dfrac{1}{\log(c/h(p_{G},p_{G'})(\log(1/h(p_{G},p_{G'})))^{d/4})}\right)^{1/2}\biggr)
 \end{eqnarray}
 for some constant $C$ dependent on the dimension and known covariance matrix.
\end{theorem}
\noindent
 The proof of Theorem~\ref{theorem: gaussian_excess} is in Appendix~\ref{ssection: proof gaussian excess mass}. The key technical novelty of the proof lies in the idea of convolving the mixing measures with a mollifier which is exponentially integrable while its Fourier transform is smoother than the Gaussian location kernel. This helps to smoothly transition the problem of bounding distances on mixing measures to the Fourier transform domain of corresponding mixture densities. We make a few  comments about the above theorem. 

(i) The upper bound on the RHS of equation~\eqref{eq: orlicz-Wasserstein-hellinger} depends on a power of log-Hellinger distance between the corresponding mixture densities. This strengthens the result in Theorem 2 of~\cite{Nguyen-13}, who obtained a $(\log(1/h))^{-1/2}$ upper bound for $W_2(G,G')$. The result in Theorem~\ref{theorem: gaussian_excess} is obtained in terms of Orlicz-Wasserstein distances relative to an exponential convex function, thus lending it more flexibility.

(ii) The key object to obtaining this result is to find a suitable mollifier $Z_{\delta}$, which we choose as $c \dfrac{1}{\delta}(\int \exp(-itx/\delta) \exp( -t^4) \mathrm{d}t)^2$ with $c$ being the constant of proportionality for the proof of Theorem~\ref{theorem: gaussian_excess}. However, we believe a more refined choice of mollifier can yield sharper estimates on the RHS of equation~\eqref{eq: orlicz-Wasserstein-hellinger}.

(iii) The result is obtained with exact computation of the involvement of $\bar{\theta}$. Therefore, it can also be used for posterior contraction rates with sieve priors, although for this work we study  only compactly supported priors.
\paragraph{Outline of proof of Theorem~\ref{theorem: gaussian_excess}:} 
Here, we provide a proof strategy for Theorem~\ref{theorem: gaussian_excess}, which relies on the following triangle inequality with Orlicz-Wasserstein distance between $G$ and $G'$:
\begin{eqnarray}
     W_{\Phi}(G,G') \leq  W_{\Phi}(G, G \ast Z_{\delta,d}) + W_{\Phi}(G', G' \ast Z_{\delta,d})  + W_{\Phi}(G \ast Z_{\delta,d}, G' \ast Z_{\delta,d}), \label{eq:inequality_Orlicz_Wasserstein_bound}
\end{eqnarray}
where $Z_{\delta,d}(x_1,\dots,x_d) : = \prod_{i=1}^d \zeta_{\delta}(x_i)$ and $\zeta_{\delta}(x) : = c \dfrac{1}{\delta}(\int \exp(-itx/\delta) \exp( -t^4) \mathrm{d}t)^2$, with $c$ being the constant of proportionality. To control both $W_{\Phi}(G, G \ast Z_{\delta,d})$ and $W_{\Phi}(G', G' \ast Z_{\delta,d})$, we use the following lemma:
\begin{lemma}
\label{lemma:convolution_bound_orlicz_Wasserstein}
Assume that $\nu_2 = \nu_1 \ast Z_{\delta,d}$ where $\nu_{1}$ is a given probability measure on $(\mathbb{R}^d,\|\cdot\|)$. Furthermore, suppose that $\Phi(x) \leq \exp(x^{\alpha})-1$ for some $1 < \alpha <4/3$. Then, there exists universal constant $C_{\alpha}$ depending only on $\alpha$ such that
    \begin{eqnarray*}
         W_{\Phi}(\nu_1, \nu_2) \leq C_{\alpha} \delta.
    \end{eqnarray*}
\end{lemma}
\noindent
The proof of Lemma~\ref{lemma:convolution_bound_orlicz_Wasserstein} is in Appendix~\ref{subsec:proof:lemma:convolution_bound_orlicz_Wasserstein}. For the final term $W_{\Phi}(G \ast Z_{\delta,d}, G' \ast Z_{\delta,d})$, we can upper bound it using the following result:
\begin{lemma}
\label{lemma: orlicz_wasserstein difference}
Let $\nu_1,\nu_2$ be probability measures on $(\mathbb{R}^d,\|\cdot\|)$ and let $\Phi$ be a convex function satisfying conditions (i) and (ii) in Definition~\ref{definition: orlicz function}. 
Then, we obtain that
\begin{eqnarray}
W_{\Phi}(\nu_1,\nu_2) &\leq 2 \inf \{\lambda \in \mathbb{R}^{+}: \\ &\int _{\mathbb{R}^d }\Phi (\|x\|/\lambda)\,d|\nu_1(x)- \nu_2(x)|\leq 1\}. \nonumber
\end{eqnarray}
\end{lemma}
\noindent
The proof of Lemma~\ref{lemma: orlicz_wasserstein difference} is in Appendix~\ref{subsec:proof:lemma: orlicz_wasserstein difference}. Using triangle inequality and Lemmas~\ref{lemma:convolution_bound_orlicz_Wasserstein} and~\ref{lemma: orlicz_wasserstein difference}, we obtain
\begin{eqnarray}
     W_{\Phi}(G,G') \precsim \delta + \inf \{\lambda \in \mathbb{R}^{+}: \int _{\mathbb{R}^d }\Phi (\|x\|/\lambda)\cdot |(G-G') \ast Z_{\delta,d}(x)|\mathrm{d}x  \leq 1\}. \nonumber
\end{eqnarray}
We then decompose the integral with respect to $\mathbb{R}^{d}$ into two integrals: one with respect to $\|x\| \leq M$ and one with respect to $\|x\| > M$, and after some algebraic manipulations, we have
\begin{align*}
    & \hspace{-2 em} \inf \biggr\{\lambda \in \mathbb{R}^{+}: 
     \int _{\mathbb{R}^d }\Phi (\|x\|/\lambda)\cdot |(G-G') \ast Z_{\delta,d}(x)|\mathrm{d}x  \leq 1\biggr\} \\
    & \precsim \dfrac{M}{ \log(C/(  h(p_{G},p_{G_{0}}) \exp(\alpha^2d \delta^{-4})M^{d/2}))} + \dfrac{(d \bar{\theta})^{5/4}}{\log(3/2) M^{1/4}} + \frac{\delta^{5/4}}{M^{1/4}},
\end{align*}
for any $M > 0$ where $C$ is some universal constant. Collecting these results leads to
\begin{align*}
    W_{\Phi}(G,G') &\precsim \inf_{\delta, M} \biggr\{ \delta +  \dfrac{M}{ \log(C/(  h(p_{G},p_{G_{0}}) \exp(\alpha^2d \delta^{-4})M^{d/2}))} + \dfrac{(d \bar{\theta})^{5/4}}{\log(3/2) M^{1/4}} + \frac{\delta^{5/4}}{M^{1/4}} \biggr\}.
\end{align*}
Solving the minimization problem, we obtain the conclusion of Theorem~\ref{theorem: gaussian_excess}. 

In the next section, we use Theorem~\ref{theorem: gaussian_excess} to establish posterior contraction bounds of parameter estimating in Dirichet Process Gaussian mixtures.

\subsection{Posterior contraction with Orlicz Wasserstein distances}
\label{ssection:posterior_contraction_gaussian}
On the parametric estimation front, ~\cite{Nguyen-13, Guha-MTM-19, Lin-Gaussian} establish logarithmic rates for estimating mixing measures in Dirichlet Process Gaussian mixtures. While~\cite{Nguyen-13} establishes an approximately $\log(n)^{-1/2}$ rate of contraction relative to the $W_2$ metric, more recently,~\cite{Lin-Gaussian} establish minimax type $\approx \log(n)$ rates relative to the $W_1$ metric. Putting the results in context with Lemma~\ref{lemma: orlicz_bound}, both those results imply, 
$\sum_j p_j\mathbbm{1}_{\|\theta_j-\theta^0_i\|>\eta \text{ for all } i}\approx \log(n)$, meaning the mass of posterior sample atoms in the region of parameter space not populated by atoms of the true (data-generating) mixing measure decays logarithmically. This puts the use of DPGMMs for clustering in a negative light.

In this section, we show that a much stronger almost polynomial rate can be established for this objective, facilitated by the use of Orlicz-Wasserstein metrics. To facilitate our presentation, we consider the following notation.
\begin{align}
    \excessmass_{\eta}(\Theta, r) : = \biggr\{G &= \sum p_i \delta_{\theta_i} \in \overline{\Gcal}(\Theta_{n1})  : \sum_j p_j \mathbbm{1}_{\{ \|\theta_j-\theta^0_i\|> \eta \text{ for all } i \}} \geq r \biggr\}. \label{eq:excess_mass_set} 
\end{align}

$\excessmass_{\eta}(\Theta, r)$ here denotes the set of mixing measures which devote at least $r$ probability mass to atoms which are away from the atoms of $G_0$ by distance $\eta$. To study the contraction of mixing measure of DPGMMs, we impose the following assumption on the base distribution $H$.


(P.1) The base distribution $H$ is supported on $\Theta= [-\bar{\theta},\bar{\theta}]^{d}$, and absolutely continuous with respect to the Lebesgue measure $\mu$ on $\Theta$ and admits a density function $g(\cdot)$. Also, $H$ is approximately uniform, i.e., $\min_{\theta \in \Theta} g(\theta) > \dfrac{c_0}{\mu(\Theta)} > 0$.

Let $f_1(n,d) := (\log(n)/(d+2)- \log(\log n))^{-1/8}$.
\begin{theorem}
\label{theorem:posterior_gaussian_excess}
Given the Dirichlet Process Gaussian mixture models~\eqref{eq:DPMM}, if $\Phi$ satisfies the assumptions in Theorem~\ref{theorem: gaussian_excess}, then for any $\eta >0$ the following holds:
 \begin{eqnarray}
    \Pi_n \biggr(G \in \overline{\Gcal}(\Theta)&: W_{\Phi}(G,G_0) \geq  f_1(n,d) \; \bigg| \; X_{1:n}\biggr) \overset{P_{G_0}^{n}}{\rightarrow} 0. \nonumber
\end{eqnarray}
\end{theorem}
\noindent
 The proof of Theorem~\ref{theorem:posterior_gaussian_excess} is in Appendix~\ref{subsec:proof:theorem:gaussian_excess}. The following result is a simple corollary of Theorem~\ref{theorem:posterior_gaussian_excess}.

 \begin{corollary}
 \label{corollary:gaussian_excess}
 Given all the assumptions in Theorem~\ref{theorem:posterior_gaussian_excess},
\begin{eqnarray}
    \Pi_n \biggr(G \in \excessmass_{\eta} \left(\Theta, 2\exp\left( \dfrac{-\eta\log(n)^{1/8}}{(d+2)}\right)\right) \; \bigg| \; X_{1:n} \biggr) \overset{P_{G_0}}{\rightarrow} 0.
\end{eqnarray}  
 \end{corollary}

The proof of Corollary~\ref{corollary:gaussian_excess} is in Appendix~\ref{proof:corollary}. 

\vspace{0.5 em}
\noindent
\textbf{Remarks:}
(i) Corollary~\ref{corollary:gaussian_excess} suggests that if $\eta$ can be chosen sufficiently small so that each $\eta$-neighborhood contains at most one true atom, Gaussian mixture models can be useful choices in clustering as well since outlier atoms vanish at almost polynomial rates. 

(ii) We believe the rate of contraction can be optimized further with a more refined choice of $\Phi(\cdot)$, however, we make no such attempts in this work. Corollary~\ref{corollary:gaussian_excess} reveals the power of Orlicz-Wasserstein distances for Gaussian mixture models. On the other hand, this exponential choice of $\Phi$ does not improve on the bound for heavy tailed kernels such as Laplace location mixtures. 

We show in this section that Orlicz-Wasserstein metrics provide strong theoretical guarantees for mixing measures. This raises the natural question as to how such a metric can be computed for arbitrary choices of $\Phi$. We provide some guidance in that regard in the following section.

\subsection{Computation of the Orlicz-Wasserstein}
\label{ssection:computation}

\begin{figure*}[htp]
\centering
\subfigure[Standard Wasserstein $W_1$]{\includegraphics[scale=0.45]{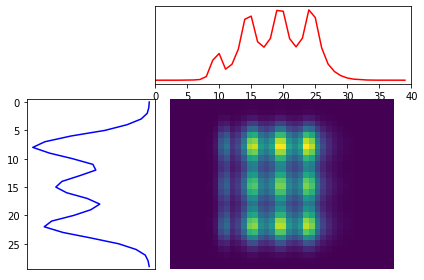}}
  \subfigure[Orlicz-Wasserstein with exponential $\Phi$]{\includegraphics[scale=0.45]{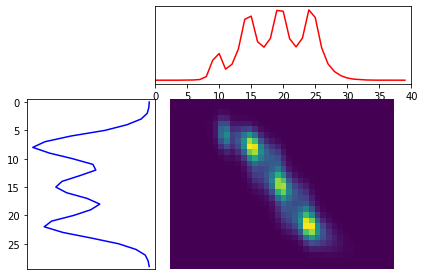}}

\caption{Transportation plans. (a) Entropic OT produces more global plans and is unable to capture local structure of mass transfers. (b) Entropic Orlicz-Wasserstein penalizes mass transfers over large distances}
\label{fig:sim-performance}
\end{figure*}

 In practice, the Euclidean Wasserstein distance is computed for samples of the respective distributions. The exact computation turns out to be a linear programming problem which scales to the order of $O(n^3 \log(n))$, where $n$ is the combined sample size of the two sampling distributions for which the distance is being calculated.~\cite{Cuturi-Sinkhorn-13} shows that using entropic regularization this can be drastically improved to $O(n^2)$~\cite{altschuler2017near, lin2019efficient, Lin-2019-Acceleration}. Further speed-ups and easiness of computation via the use of dual formulation of the entropic regularization has been explored by the works of~\cite{Seguy-etal-2017,genevay_large_scale,Genevay-thesis}. Here we consider the entropic regularized version of the Orlicz-Wasserstein metrics.

\vspace{0.5 em}
\noindent
\textbf{Computational procedure:} In that respect, we consider solving the following problem as a surrogate to equation~\eqref{eq:orlicz-Wasserstein}.
\begin{eqnarray}
\label{eq:surrogate_orlicz-Wasserstein}
     W^{\lambda}_{\Phi}(\nu_1,\nu_2)&:= \inf_{\nu \in \mathcal{Q}(\nu_1,\nu_2)} A_\Phi(\nu_1,\nu_2), \\
     P^{\lambda}_{\Phi}(\nu_1,\nu_2)&:= \arg\inf_{\nu \in \mathcal{Q}(\nu_1,\nu_2)} A_\Phi(\nu_1,\nu_2),
\end{eqnarray}
 where $A_\Phi(\nu_1,\nu_2):=\inf \{\eta \in \mathbb{R}^{+}: 
 \int _{\mathbb{R}^d \times \mathbb{R}^d}\Phi (\|x-y\|/\eta)\,d\nu(x,y) -(1/\lambda) (H(\nu)) \leq 1\}$ with  $H(\mu)$  used to denote the Shannon entropy of distribution $\mu$. To obtain solutions for equation~\eqref{eq:surrogate_orlicz-Wasserstein}, we resort to using outputs from Sinkhorn divergence computations.

\begin{algorithm}[h]
\caption{Computing Orlicz Wasserstein distances between two discrete probability measures} 
\label{algo:compute_orlicz}
\begin{algorithmic}[1]
\STATE \textbf{Input} M, $\lambda, \vec{r}, \vec{c}, \epsilon$.
\STATE \textbf{Output} $W^{\lambda}_{\Phi}(\nu_1,\nu_2)$.
\STATE I = $(\vec{r}>0)$; $\vec{r} = \vec{r}(I)$; $M=M(I,:)$;
\STATE $x_{\text{upp}}=\max(M)/\Phi^{-1}(1)$, \\
$x_{\text{low}}=[S(M,\lambda,\vec{r},\vec{c}) +\dfrac{1}{2\lambda}(H(\vec{r})+H(\vec{c}))]/\Phi^{-1}(1+\dfrac{1}{\lambda}(H(\vec{r})+H(\vec{c})) $
\STATE $fx_{\text{upp}}=S(\Phi(M/x_{\text{upp}}),\lambda,\vec{r},\vec{c})$,$fx_{\text{low}}=S(\Phi(M/x_{\text{low}}),\lambda,\vec{r},\vec{c})$.

\STATE while $|x_{\text{low}}-x_{\text{upp}}|<\epsilon$ not converged do
\STATE $x_{\text{new}}=(x_{\text{low}}*fx_{\text{upp}}-x_{\text{upp}}*fx_{\text{low}})/(fx_{\text{upp}}-fx_{\text{low}}))$.
\STATE if $x_{\text{new}}<x_{\text{upp}}$ and $x_{\text{new}}>x_{\text{low}}$ do
\STATE $fx_{\text{new}}=S(\Phi(M/x_{\text{upp}}),\lambda,\vec{r},\vec{c})$
\STATE if $fx_{\text{new}}<1$, $x_{\text{upp}}=x_{\text{new}}, fx_{\text{upp}}=fx_{\text{new}}$.
\STATE else: $x_{\text{low}}=x_{\text{new}}, fx_{\text{low}}=fx_{\text{new}}$
\STATE end if
\STATE else $x_{\text{new}}= (x_{\text{low}}+x_{\text{upp}})/2$. repeat Step 9-12.
\STATE end if
\STATE end while
\STATE return $W^{\lambda}_{\Phi}(\nu_1,\nu_2):=x_{\text{upp}}$. 
\end{algorithmic}		
\end{algorithm}

Consider two discrete probability measures, $\vec{r}$ (with $m$ atoms, $\{x_i\}_{i=1}^m$) and $\vec{c}$ (with $n$ atoms, $\{y_i\}_{i=1}^m$). Let $M_{n\times m}$ be a distance matrix such that $M_{ij}=c(x_i,y_j)$ for some cost function $c(\cdot, \cdot)$. Let $S(M,\lambda,r,c)$ be used to denote the Sinkhorn divergence optimized objective function for cost matrix $M$, regularization parameter $\lambda$ and $d(M,\lambda,r,c)=\langle S(M,\lambda,r,c),M\rangle$ be used to denote the transport cost. Algorithm~\ref{algo:compute_orlicz} defines a procedure to obtain a regularised Orlicz-Wasserstein distance between $\nu_1=\sum_i r_i\delta_{x_i}$ and $\nu_2=\sum_i c_i\delta_{y_i}$ in such a scenario by iteratively updating the value of Orlicz-Wasserstein distance until convergence. The crucial intuition behind Algorithm~\ref{algo:compute_orlicz} is that $\inf_{\nu \in \mathcal{Q}(\nu_1,\nu_2)}\int _{\mathbb{R}^d \times \mathbb{R}^d}\Phi (\|x-y\|/\eta)\,d\nu(x,y) -(1/\lambda) (H(\nu))$ is a monotonically non-increasing function of $\eta$. Therefore the solution to the Orlicz Wasserstein distance can be obtained by a binary search once upper and lower limits are known. This is rigorously explained in Proposition~\ref{prop: algo_proof} in Appendix~\ref{ssection:Entropic_OT}. 

\vspace{0.5 em}
\noindent
\textbf{Simulations settings:} We provide a demonstration of the utility of using Orlicz-Wassestein distances in Figure~\ref{fig:sim-performance}. We consider two mixing densities, $\nu_1$ on the y-axis is a 3-mixture of univariate normal distributions with means at $[3,4,5]$, common $\sigma=0.3$ and mixture weights $[0.37,0.3,0.33]$. On the other hand $\nu_2$ represented in the x-axis is a 4-mixture of univariate Laplace kernels with means at $[7,8,9,6] $, scale parameters $[0.3,0.3,0.3,0.1]]$ and mixture weights $[0.30,0.32,0.32,0.06]$. The left plot of Figure~\ref{fig:sim-performance} shows the transportation plan for output of Sinkhorn mechanism with regularisation parameter $0.01$, while the right plot shows the same for transportation plan obtained via Algorithm~\ref{algo:compute_orlicz} with $\lambda=0.01$ ($\Phi(\cdot)=\exp(\cdot/\beta)-1$, $\beta=1.1$). We have the following remarks.

\vspace{0.5 em}
\noindent
\textbf{Remark:} The entropic Orlicz-Wasserstein  procedure produces sharper transport plans. This indicates that it performs a shrinkage procedure on the space of transportation plans. This can have potential benefits towards obtaining robust plans and provide a promising direction of future research. Additionally, while entropic Euclidean Wasserstein transport plans distribute the mass of the outlier atom of $\nu_2$ (mean=6, weight= 0.06), its Orlicz-Wasserstein counterpart manages to avoid it entirely. By penalizing mass transfers over large distances, Orlicz-Wasserstein distances are able to restrict attention to localised transportation plans. This in turn helps capture the small outlier mass associated with aposteriori DPGMM samples, as seen in Section~\ref{ssection:posterior_contraction_gaussian}. 

\section{Conclusion}
In this work, we discuss the shortcomings of traditional Wasserstein metrics to perform clustering with Gaussian mixture models. We re-introduce a novel metric, called Orlicz-Wasserstein distances, in the context of estimating parameter convergence rates of hierarchical and mixture models and provide sound theoretical justifications of its ability to address the concerns associated with traditional Wasserstein distances. We also provide a theoretically sound approximate algorithm to compute the distance metric, and also show that convergence rates of Orlicz-Wasserstein distances carry over to the approximate distance. Lastly, we provide a preliminary simulation study to initiate a discussion on future research with Orlicz-Wasserstein distances. Since they allow low/high penalty on mass transfers over large distances, depending on the choice of function $\Phi$, this lends flexibility to extending mass transfers over local/global regions and consequentially may be used as a device for smoothing/sharpening standard OT plans. Combined with dimension reduction techniques this can lend usage to a number of application domains such as  anomaly detection and robust optimal transport.

\section*{Acknowledgements}
We thank Professor Fedor Nazarov and Professor Mark Rudelson for discussions on the paper.
This research is supported in part by grants NSF CAREER DMS-1351362, NSF CNS-1409303, a research
gift from Adobe Research and a Margaret and Herman Sokol Faculty Award. Nhat Ho acknowledges support from the NSF IFML 2019844 and the NSF AI Institute for Foundations
of Machine Learning.


\bibliography{Aritra}

\newpage
\appendix
\onecolumn


\begin{center}
{\bf{\Large{Supplement to ``On Excess Mass Behavior in Gaussian Mixture Models with Orlicz-Wasserstein Distances"}}}
\end{center}
In this supplementary material, we present proofs of key results in Appendix~\ref{sec:proofs} and proofs of lemmas in Appendix~\ref{sec:proof_lemma}. We then provide theoretical guarantee for the algorithm to compute the entropic regularized Orlicz-Wasserstein in Appendix~\ref{ssection:Entropic_OT}. 

\section{Proofs of key results}
\label{sec:proofs}
\paragraph{Notation revisited} For any function $g:\mathcal{X} \to \mathbb{R}$, we denote $\widetilde{g}(\omega)$ as the Fourier transformation of function $g$. Given two densities $p, q$ (with respect to the Lebesgue measure $\mu$), the squared Hellinger distance is given by  $h^{2}(p,q)= {\displaystyle (1/2) \int (\sqrt{p(x)}-\sqrt{q(x)})^{2}\textrm{d}\mu(x)}$. 
For any metric $d$ on $\Theta$, we define the open ball of $d$-radius $\epsilon$ around $\theta_0 \in \Theta$ as $B_d(\epsilon,\theta_0)$. Additionally, the expression $a_{n} \gtrsim b_{n}$ will be used to denote the inequality up to a constant multiple where the value of the constant is independent of $n$. We also denote $a_{n} \asymp b_{n}$ if both $a_{n} \gtrsim b_{n}$ and $a_{n} \lesssim b_{n}$ hold. Furthermore, we denote $A^{c}$ as the complement of set $A$ for any set $A$ while $B(x,r)$ denotes the ball, with respect to the $l_2$ norm, of radius $r > 0$ centered at $x \in \mathbb{R}^{d}$. The expression $D(\epsilon,\mathscr{P},d)$ used in the paper denotes the $\epsilon$-packing number of the space $\mathscr{P}$ relative to the metric $d$. $d$ is replaced by $h$ to denote the hellinger norm. Finally, we use $\text{Diam}(\Theta)= \sup \{\|\theta_1-\theta_2\|: \theta_1,\theta_2 \in \Theta\}$  to denote the diameter of a given parameter space $\Theta$ relative to the $l_{2}$ norm, $\|\cdot\|$, for elements in $\mathbb{R}^{d}$. Regarding the space of mixing measures, let $\Ecal_{k}:=\Ecal_{k}(\Theta)$ and $\Ocal_{k}:=\Ocal_{k}(\Theta)$ respectively denote the space of all mixing measures with exactly and at most $k$ support points, all in $\Theta$. Additionally, denote $\Gcal : = \Gcal(\Theta) = \mathop{\cup} \limits_{k \in \mathbb{N}_{+}}{\Ecal_{k}}$ the set of all discrete measures with finite supports on $\Theta$. Moreover, $\overline{\Gcal}(\Theta)$ denotes the space of all discrete measures (including those with countably infinite supports) on $\Theta$. Finally, $\mathcal{M}(\Theta)$ stands for the space of all probability measures on $\Theta$.

\subsection{Proof of Theorem~\ref{theorem: gaussian_excess}}
\label{ssection: proof gaussian excess mass}

We present the proof of Theorem~\ref{theorem: gaussian_excess} for the lower bound of Hellinger distance between mixing density functions based on Orlicz-Wasserstein metric between their corresponding mixing measures.

In this proof, we denote $a \lesssim b $ to imply that $ a\leq C\cdot b$ for a universal constant $C$ dependent on $\alpha, d,$ and $ \bar{\theta}$. Also, $f \ast g$ will denote the outcome of convolution operation on functions $f$ and $g$. Now, we consider the following density function in $\mathbb{R}$:
\begin{eqnarray}
\label{eq:mollifier_gauss_excess}
     K(x):= c\left(\int_{-\infty}^{\infty} \exp(-itx)\exp(-t^4)\mathrm{d}t\right)^2,
\end{eqnarray}
where $c$ is a proportionality constant so that $\int_{-\infty}^{\infty} K(x)\mathrm{d}x =1$. Lemma~\ref{lemma: fourier transform} shows that $K(\cdot)$ is integrable. 

Moreover, Lemma~\ref{lemma: fourier transform_2} shows that the characteristic function $\hat{K}(\cdot)$, corresponding to $K(\cdot)$ satisfies,
\begin{eqnarray}
|\hat{K}(x)|\lesssim  \exp(-(x/2)^4) \nonumber.
\end{eqnarray}
The strategy to obtain upper bounds for $W_{\Phi}(G,G')$ is to convolve $G$ with mollifiers, $Z_{\delta,d}(\cdot)$, of the form $Z_{\delta,d}(x)=\prod_{i=1}^d \dfrac{1}{\delta}K(x_i/\delta)$ for $\delta>0$, where $x = (x_1, \dots,x_d)$. In particular, by triangle inequality and following Lemma~\ref{lemma: orlicz_metric} we can write:
\begin{eqnarray}
     W_{\Phi}(G,G')\leq  W_1(G, G \ast Z_{\delta,d}) + W_{\Phi}(G', G' \ast Z_{\delta,d})+ W_{\Phi}(G \ast Z_{\delta,d}, G' \ast Z_{\delta,d}). \nonumber
\end{eqnarray}
For $\Phi(x)= \exp((7/32)x)-1$, following Lemma~\ref{lemma:convolution_bound_orlicz_Wasserstein} we find that 
\begin{eqnarray}
     W_{\Phi}(G,G \ast Z_{\delta,d}) \leq C_{\alpha} \delta. \nonumber
\end{eqnarray}
Therefore, we can write 
\begin{eqnarray}
     W_{\Phi}(G,G')\leq 2C_{\alpha} \delta + W_{\Phi}(G \ast Z_{\delta,d}, G' \ast Z_{\delta,d}). \nonumber
\end{eqnarray}
For every $M>0$, we have
\begin{align}
 W_{\Phi}(G \ast Z_{\delta,d}, G' \ast Z_{\delta,d}) & \leq  2 \inf \{\lambda \in \mathbb{R}^{+}: \int _{\mathbb{R}^d }\Phi (\|x\|/\lambda)\cdot |(G-G') \ast Z_{\delta,d}(x)|\mathrm{d}x  \leq 1\}\nonumber \\
& \leq  2  \inf \{\lambda \in \mathbb{R}^{+}:  \ s_{1} \leq 1/2 \text { and } \ s_{2}  \leq 1/2\}, \nonumber  \\
& \leq 2 \max \{\inf \{\lambda \in \mathbb{R}^{+}:  \ s_{1} \leq 1/2\}, \inf \{\lambda \in \mathbb{R}^{+}:  \ s_{2} \leq 1/2\} \}, \label{eq:key_inequality_orlicz}
\end{align}
with the first inequality following from Lemma~\ref{lemma: orlicz_wasserstein difference} and the third inequality comes from the monotonicity of function $\Phi$. Here, we denote
\begin{align*}
    s_{1} & = \int \limits_{\|x\|_2 \leq M}\Phi (\|x\|/\lambda)\cdot |(G-G') \ast Z_{\delta,d}(x)|\mathrm{d}x, \\
    s_{2} & = \int \limits_{\|x\|_2 > M}{\Phi (\|x\|/\lambda)\cdot |(G-G') \ast Z_{\delta,d}(x)|\mathrm{d}x}.
\end{align*}
We now proceed to bound $T_{1} = \inf \{\lambda \in \mathbb{R}^{+}:  \ s_{1} \leq 1/2\}$ and $T_{2} = \inf \{\lambda \in \mathbb{R}^{+}:  \ s_{2} \leq 1/2\}$.
\paragraph{Bounding for $T_{1}$:}
 Using Holder's inequality, we obtain
\begin{eqnarray}
\label{eq: orlicz_convolution_proof_first}
 & &\inf \{\lambda \in \mathbb{R}^{+}: \int \limits_{\|x\|_2 \leq M}\Phi (\|x\|/\lambda)\cdot |(G-G') \ast Z_{\delta,d}(x)|\mathrm{d}x \leq 1/2 \} \nonumber \\ & \leq & \inf\{ \lambda >0: \int \limits_{\|x\| \leq M}{\exp((\|x\|/\lambda)^{\beta})\cdot |(G-G') * Z_{\delta,d}(x)|\mathrm{d}x} \leq 3/2\} \nonumber \\
  & \leq & \inf\biggr\{ \lambda >0: \biggr(\int \limits_{\|x\| \leq M} \exp((M/\lambda)^{\beta})\mathrm{d}x\biggr)^{1/2} \biggr(\int \limits_{\|x\| \leq M} |(G-G') * Z_{\delta,d}(x)|^{2}\mathrm{d}x\biggr)^{1/2} \leq 3/2 \biggr\} \nonumber \\
 &\leq &  \inf \biggr\{ \lambda >0: \dfrac{\pi^{d/4}}{\sqrt{(\frac{d}{2}+1)\Gamma(d/2)}}M^{d/2}\exp((M/\lambda)^{\beta}) \|(G-G') *Z_{\delta,d}(x)\|_{2} \leq 3/2 \biggr\}   \nonumber\\
  &=& \dfrac{M}{ (\log(c_d/(\|(G-G') *\zeta_{\delta,d}\|_{2}M^{d/2})))^{1/\beta}}.
\end{eqnarray}
Since $f$ is Gaussian distribution, we have $\tilde{f}(\omega) \geq c_f \exp(-\alpha \sum_{i=1}^d \omega_i^2)$ for some $c_f,\alpha>0$. Given that inequality, we find that
\begin{eqnarray}
\|(G-G') * Z_{\delta,d}\|_{2}^{2} & = & \int |\widetilde{G} - \widetilde{G'}|^{2}(\omega)|\widetilde{K}_{\delta,d}(\omega)|^{2}d\omega = \int |\widetilde{f}(\widetilde{G} - \widetilde{G'})|^{2}(\omega)\dfrac{|\widetilde{K}_{\delta,d}(\omega)|^{2}}{|\widetilde{f}(\omega)|^{2}}d\omega \nonumber \\
& \leq & \|p_{G} - p_{G'}\|_{2}^{2}\sup \limits_{\omega \in \mathbb{R}^{d}}\dfrac{|\widetilde{K}_{\delta,d}(\omega)|^{2}}{|\widetilde{f}(\omega)|^{2}} \nonumber\\
& \leq & 4 \|f\|_{\infty} h^{2}(p_{G},p_{G_{0}}) \sup \limits_{\omega \in \mathbb{R}^{d}} \biggr\{\frac{1}{c_f^2} \cdot  \prod \limits_{i=1}^{d} \exp(-\delta^{4}|\omega_i|^{4})\exp(2\alpha |\omega_i|^2) \biggr\}. \nonumber
\end{eqnarray}
By taking derivatives, we obtain the maximum as
\begin{align*}
\sup \limits_{\omega_i \in \mathbb{R}} \biggr\{ \exp(-\delta^{4}|\omega_i|^{4})\exp(2\alpha |\omega_i|^2) \biggr\}=\exp(\alpha^2/\delta^4).
\end{align*}
Plugging these results into equation~\eqref{eq: orlicz_convolution_proof_first} leads to 
\begin{align}
 \inf \{\lambda \in \mathbb{R}^{+}: \int \limits_{\|x\|_2 \leq M}\Phi (\|x\|/\lambda)\cdot |(G-G') \ast Z_{\delta,d}(x)|\mathrm{d}x \leq 1/2 \} & \nonumber \\  & \hspace{-12 em} \leq \dfrac{M}{ (\log(c/(  h(p_{G},p_{G_{0}}) \exp(\alpha^2d \delta^{-4})M^{d/2})))^{1/\beta}} \label{eq: orlicz_convolution_proof_first_first}
\end{align}
for some universal constant $c$.
\paragraph{Bounding for $T_{2}$:}
For any $M>0$, we denote 
\begin{eqnarray}
\label{eq:convex_phi}
k' & = & \inf \{\lambda \in \mathbb{R}^{+}: \mathbb{E}_{X \sim (G-G')}(\Phi (\|X\|^{5/4}/\lambda M^{1/4})\leq 1/2\}, \nonumber \\ 
k''& = &\inf \{\lambda \in \mathbb{R}^{+}: \mathbb{E}_{ Y \sim Z_{\delta,d}}(\Phi (\|Y\|^{5/4}/\lambda M^{1/4})\leq 1/2\}.
\end{eqnarray}
Then, by the convexity of $\Phi$ we have 
\begin{align*}
   \inf \{\lambda \in \mathbb{R}^{+}: \mathbb{E}_{X \sim G-G', Y \sim Z_{\delta,d}}(\Phi (\|X+Y\|^{5/4}/ \lambda M^{(1/4)})\leq 1/2\} \leq 2^{1/4}(k'+k").
\end{align*}
The above inequality is because of the following inequalities:
\begin{eqnarray}
     & &\mathbb{E}_{X \sim G-G', Y \sim Z_{\delta,d}}(\Phi (\|X+Y\|^{5/4}/2^{1/4}(k'+k")M^{(1/4)}) \nonumber \\ &\leq &\mathbb{E}_{X \sim G-G', Y \sim Z_{\delta,d}}(\Phi (2^{1/4}(\|X\|^{5/4}+\|Y\|^{5/4})/2^{1/4}(k'+k")M^{(1/4)})) \nonumber\\ &= &\mathbb{E}_{X \sim G-G', Y \sim Z_{\delta,d}} \left(\Phi \left(\left(k'\|X\|^{5/4}+\|Y\|^{5/4}\right)/(k'+k")M^{(1/4)}\right) \right) \nonumber \\ &\leq &\mathbb{E}_{X \sim G-G', Y \sim Z_{\delta,d}} \left(\Phi \left(\dfrac{k'}{k'+k"}\left(\dfrac{\|X\|^{5/4}}{k'M^{(1/4)}}\right)+ \dfrac{k"}{k'+k"}\left(\dfrac{\|Y\|^{5/4}}{k" M^{(1/4)}}\right)\right) \right)\nonumber \\ &\leq &\mathbb{E}_{X \sim G-G', Y \sim Z_{\delta,d}} \dfrac{k'}{k'+k"} \Phi \left(\dfrac{\|X\|^{5/4}}{k'M^{(1/4)}} \right)+ \dfrac{k"}{k'+k"}\Phi\left(\dfrac{\|Y\|^{5/4}}{k" M^{(1/4)}}\right) \leq \dfrac{1}{2}.
\end{eqnarray}
The first inequality follows from $\|a+b\|^P \leq 2^{p-1}(\|a\|^p+\|b\|^p)$. The second last inequality follows from convexity of $\Phi$ and the final inequality follows from equation~\eqref{eq:convex_phi}. Therefore, we obtain that
\begin{eqnarray}
\label{eq:s2_gaussian_excess}
     \inf \{\lambda \in \mathbb{R}^{+}:  & & \int \limits_{\|x\|_2 > M}{\Phi (\|x\|/\lambda)\cdot |(G-G') \ast Z_{\delta,d}(x)|\mathrm{d}x}\leq 1/2\} \nonumber \\ & \leq & \inf \{\lambda \in \mathbb{R}^{+}: \int \limits_{\|x\|_2 > M}{\Phi (\|x\|^{5/4}/\lambda M^{(1/4)})\cdot |(G-G') \ast Z_{\delta,d}(x)|\mathrm{d}x}\leq 1/2\} \nonumber \\ & \leq & \inf \{\lambda \in \mathbb{R}^{+}: \mathbb{E}_{X \sim G-G', Y \sim Z_{\delta,d}}(\Phi (\|X+Y\|^{5/4}/\lambda M^{(1/4)})\leq 1/2\}\nonumber \\   & \lesssim &  \inf\{ \lambda >0: \int \limits_{\mathbb{R}^d}{\exp((\|x\|^{5/4}/\lambda M^{1/4})^{\beta})\cdot |(G-G') (x)|\mathrm{d}x} \leq 3/2\}\nonumber\\ 
     & + &\inf\{ \lambda >0: \int \limits_{\mathbb{R}^d}{\exp((\|x\|^{5/4}/\lambda M^{1/4})^{\beta})\cdot | Z_{\delta,d}(x)|\mathrm{d}x} \leq 3/2\}\nonumber\\
     & \lesssim& \dfrac{(d \bar{\theta})^{5/4}}{ M^{1/4}} + C \delta^{5/4}/M^{1/4},
\end{eqnarray}
where $C= \inf\{ \lambda >0: \int \limits_{\mathbb{R}^d}{\exp((\|x\|^{5/4}/\lambda)^{\beta})\cdot | K_{1,d}(x)|\mathrm{d}x} < \infty $ as $K_{1,d}(x) \sim O(\exp(-|x|^{4/3}))$ for large $|x|$, by Lemma~\ref{lemma: fourier transform}. Hence, using these results we get
 \begin{eqnarray}
      \label{eq:s2_gaussian_excess_extend}
      W_{\Phi}(G ,G') &\lesssim& \delta+ \max\biggr\{\dfrac{(d \bar{\theta})^{5/4}}{ M^{1/4}} + C \delta^{5/4}/M^{1/4} , \dfrac{M}{ (\log(c/(  h(p_{G},p_{G_{0}}) \exp(\alpha^2d \delta^{-4})M^{d/2})))^{1/\beta}}\biggr\} \nonumber \\
      & \leq & \delta+ \dfrac{(d \bar{\theta})^{5/4}}{ M^{1/4}} + C \delta^{5/4}/M^{1/4} + \dfrac{M}{ (\log(c/(  h(p_{G},p_{G_{0}}) \exp(\alpha^2d \delta^{-4})M^{d/2})))^{1/\beta}}. 
      \end{eqnarray}
 Choosing $M=(\log(1/h(p_{G},p_{G_{0}})))^{1/2}$ and $\delta= \dfrac{2\alpha^2}{\log(1/h(p_{G},p_{G_{0}}))}$ in equation~\eqref{eq:s2_gaussian_excess_extend} we obtain,
 \begin{eqnarray}
 W_{\Phi}(G,G') \lesssim (\log(1/h(p_{G},p_{G_{0}})))^{-1} &+& \dfrac{(d \bar{\theta})^{5/4}}{(\log(1/h(p_{G},p_{G_{0}})))^{1/8}} + \left(\dfrac{1}{\log(1/h(p_{G},p_{G_{0}}) )}\right)^{11/8} \nonumber \\ & + &\left(\dfrac{1}{\log(c/h(p_{G},p_{G_{0}})(\log(1/h(p_{G},p_{G_{0}})))^{d/4})}\right)^{(1/\beta)-(1/2)}
 \end{eqnarray}
 As a consequence, we obtain the conclusion of the theorem.
\subsection{Proof of
Theorem~\ref{theorem:posterior_gaussian_excess}}
\label{subsec:proof:theorem:gaussian_excess}

The proof of this result follows by an application of  Lemma~\ref{lemma:Prior_mass_DPMM},~\ref{lemma:Hellinger_metric_entropy} and~\ref{lemma:M_compute} in combination with Theorem 2.1 in~\cite{Ghosal-Ghosh-vanderVaart-00}. To facilitate the presentation, we break the proof into several steps.
\paragraph{Step 1:} First we compute the contraction rate relative to the Hellinger metric, i.e., assume that 
\begin{align*}
\frac{\bar{\theta}^{d}}{\epsilon^{d+2}_n} \log\left(\frac{\bar{\theta}}{\epsilon_n}\right) = o(n) \ \ \text{and} \ \ n \epsilon^2_n \rightarrow \infty.
\end{align*}
Then we show that \begin{eqnarray}
    \Pi_n(G \in \overline{\Gcal}(\Theta): h(p_G,p_{G_0}) \geq L\epsilon_n| X_1,\ldots, X_n) \overset{P_{G_0}}{\rightarrow} 0.
\end{eqnarray}
We apply Theorem 7.1 in~\cite{Ghosal-Ghosh-vanderVaart-00}, with $\epsilon=L\epsilon_n$ and $D(\epsilon)=\exp\left(c_1 \left(\dfrac{\bar{\theta}}{\sqrt{\lambda_{\text{min}}}\epsilon_n}\right)^d \log\left(e + \dfrac{32e \bar{\theta}^2}{\lambda_{\text{min}}\epsilon_n^2}\right)\right)$, where $L \geq 2$ is a large constant to be chosen later and $c_1$ is the constant in equation~\eqref{eq:hellinger_entropy_gaussian}. Lemma~\ref{lemma:Hellinger_metric_entropy} shows the validity of this choice of $D(\epsilon)$. Then there exists a test function $\phi_n$ that satisfies
\begin{eqnarray}
     \label{eq:ghosal_test_function}
     P_{G_0}^n\phi_n & \leq & \exp\left(c_1 \left(\dfrac{\bar{\theta}}{\sqrt{\lambda_{\text{min}}}\epsilon_n}\right)^d \log\left(e + \dfrac{32e \bar{\theta}^2}{\lambda_{\text{min}}\epsilon_n^2}\right)\right) \nonumber\\
     & & \hspace{3 em} \times \exp(-KnL^2\epsilon_n^2)\dfrac{1}{1-\exp(-KnL^2\epsilon_n^2)} ,\nonumber\\
     \sup_{G \in \overline{\Gcal}(\Theta): h(p_G,p_{G_0}) \geq L\epsilon_n} P_G^n (1-\phi_n) & \leq &\exp(-KnL^2\epsilon_n^2).
\end{eqnarray}
Now, we have
\begin{eqnarray}
\label{eq:type1error_gaussian}
     & &\mathbb{E}_{P_{G_0}}\Pi_n(G \in \overline{\Gcal}(\Theta): h(p_G,p_{G_0}) \geq L\epsilon_n| X_1,\ldots, X_n)\phi_n \nonumber \\ 
     & & \hspace{6 em}\leq P_{G_0}^n\phi_n \leq 2 \exp\left(c_1 \left(\dfrac{\bar{\theta}}{\sqrt{\lambda_{min}}\epsilon_n}\right)^d \log\left(e + \dfrac{32e \bar{\theta}^2}{\lambda_{min}\epsilon_n^2}\right) -KnL^2\epsilon_n^2\right).
\end{eqnarray}
Based on computation with the posterior,
\begin{eqnarray}
    & & \Pi_n(G: h(p_G,p_{G_0} \geq \epsilon_n)|X_1,\ldots,X_n)(1-\phi_n)  \nonumber\\ & &   \hspace{5 em} =\dfrac{\bigintsss_{G\in \overline{\Gcal}(\Theta) : h(p_G,p_{G_0}) \geq \epsilon_n} \prod_{i=1}^n \dfrac{p_G(X_i)}{p_{G_0}(X_i)} \mathrm{d}\Pi_n(G) (1-\phi_n)}{\bigintsss_{G\in \overline{\Gcal}(\Theta)} \prod_{i=1}^n \dfrac{p_G(X_i)}{p_{G_0}(X_i)} \mathrm{d}\Pi_n(G)} \nonumber \\
     & & \hspace{5 em} \leq  \dfrac{\bigintsss_{G\in \overline{\Gcal}(\Theta) : h(p_G,p_{G_0}) \geq \epsilon_n} \prod_{i=1}^n \dfrac{p_G(X_i)}{p_{G_0}(X_i)} \mathrm{d}\Pi_n(G) (1-\phi_n)}{\bigintsss_{G\in \overline{\Gcal}(\Theta) : K(p_{G_0},p_G) \lesssim \epsilon^2_n, K_2(p_{G_0},p_G) \lesssim \epsilon^2_n (\log (M/\epsilon_n))^2 } \prod_{i=1}^n \dfrac{p_G(X_i)}{p_{G_0}(X_i)} \mathrm{d}\Pi_n(G)}  \nonumber,
     \end{eqnarray}
where $M=\exp(d\lambda_{min}^{-1} (  5\bar{\theta}_0^2 + 4\bar{\theta}^{2}))$, with $\lambda_{min}$ being the minimum eigenvalue of $\Sigma$.

\paragraph{Step 1.1:}
In this step we show that 
\begin{eqnarray}
\label{eq:KL_integral}
& &{\bigintsss_{G\in \overline{\Gcal}(\Theta) : K(p_{G_0},p_G) \lesssim \epsilon^2_n, K_2(p_{G_0},p_G) \lesssim \epsilon^2_n (\log (M/\epsilon_n))^2 } \prod_{i=1}^n \dfrac{p_G(X_i)}{p_{G_0}(X_i)} \mathrm{d}\Pi_n(G)} \nonumber\\
     & & \hspace{5 em} \gtrsim \exp(-(1+C)n \lambda_{min}\epsilon_n^2)\frac{\Gamma(\gamma) (c_0\gamma\pi^{d/2})^{D}}{(2\Gamma(d/2+1))^{D}(2D)^{D-1}} 
		\left(\frac{\sqrt{\lambda_{min}}\epsilon_n}
		{2\bar{\theta}}\right)^{2(D-1) +dD} \\
		& & \hspace{ 25 em} \text{with $p_{G_0}^n$ probability}\rightarrow 1, \nonumber
\end{eqnarray}
for all $C>0$ and $\epsilon_n>0$ is sufficiently small, where $D = D(\sqrt{\lambda_{min}}\epsilon_n, \Theta, \|.\|) \approx \left(\dfrac{\bar{\theta}}{\epsilon_n} \right)^d$ stands for the maximal $\sqrt{\lambda_{min}}\epsilon_n$-packing number for $\Theta$ under $\|.\|$ norm, and $\Gamma(\cdot)$ is the gamma function. 
First we show that
\begin{eqnarray}
\label{eq:K_Wasserstein}
 & & \{G \in \overline{\Gcal}(\Theta) : W_2(G,G_0) \lesssim \sqrt{\lambda_{min}}\epsilon_n\} \nonumber \\
 & & \hspace{5 em} \subset \{G\in \overline{\Gcal}(\Theta) : K(p_{G_0},p_G) \lesssim \epsilon^2_n, K_2(p_{G_0},p_G) \lesssim \epsilon^2_n (\log (M/\epsilon_n))^2 \},  
\end{eqnarray}
for $\epsilon_n$ sufficiently small.

Since $\bigintss \dfrac{(p_{G_0}(x))^2}{p_G(x)}\mu(\mathrm{d}x) \leq M$ by Lemma~\ref{lemma:M_compute}, it follows by an application of Theorem 5 in~\cite{Wong-Shen-95} that for $\epsilon_n < 1/2(1-e^{-1})^2$, 
\begin{eqnarray}
   h(p_G,p_{G_0}) \lesssim \epsilon_n^2 \implies K_2(p_{G_0},p_G) \lesssim \epsilon^2_n (\log (M/\epsilon_n))^2. \nonumber
\end{eqnarray}
Following Example 1 in~\cite{Nguyen-13}, $h^2(p_G,p_{G_0}) \leq \dfrac{W_2^2(G,G_0)}{8 \lambda_{min}}$ for Gaussian location mixtures.

Similarly, from~\cite{Nguyen-13} it also follows that $K(p_G,p_{G_0}) \leq \dfrac{W_2^2(G,G_0)}{2 \lambda_{min}} $. Combining the above displays, equation~\eqref{eq:K_Wasserstein} follows.

Following Lemma 8.1 in~\cite{Ghosal-Ghosh-vanderVaart-00}, for every $C, \epsilon, M >0$ and any measure $\Pi$ on the set $ \{G\in \overline{\Gcal}(\Theta) : K(p_{G_0},p_G) \lesssim \epsilon^2_n, K_2(p_{G_0},p_G) \lesssim \epsilon^2_n (\log (M/\epsilon_n))^2 \}$, we have,
\begin{eqnarray}
\label{eq:Ghosal_ghosh_VDV_KL_result}
    P_{G_0}^n \left(\int \prod_{i=1}^n \dfrac{p_G(X_i)}{p_{G_0}(X_i)} \mathrm{d}\Pi_n(G) \leq \exp(-(1+C)n\epsilon^2)\right)\leq \dfrac{1}{C^2n\epsilon^2(\log(M/\epsilon))^2}.
\end{eqnarray}
The result in equation~\eqref{eq:KL_integral} now follows by an application of Lemma~\ref{lemma:Prior_mass_DPMM} in combination with equations~\eqref{eq:K_Wasserstein} and~\eqref{eq:Ghosal_ghosh_VDV_KL_result} using the fact that $n\epsilon_n^2 \rightarrow \infty$.

\paragraph{Step 1.2: } Let the event in ~\eqref{eq:KL_integral} be denoted as $T_n$. Then
\begin{eqnarray} 
 \label{eq:type2error_gaussian}
    & &\mathbb{E}_{P_{G_0}}\left[\Pi_n(G: h(p_G,p_{G_0}) \geq L\epsilon_n)|X_1,\ldots,X_n)(1-\phi_n)\right] \leq P_{G_0}(T_n^C) \nonumber \\&+& P_{G_0}(T_n) \dfrac{\exp((1+C)n \lambda_{min}\epsilon_n^2)}{\frac{\Gamma(\gamma) (c_0\gamma\pi^{d/2})^{D}}{(2\Gamma(d/2+1))^{D}(2D)^{D-1}} 
		\left(\frac{\sqrt{\lambda_{min}}\epsilon_n}
		{2\bar{\theta}}\right)^{2(D-1) +dD} } \sup_{G \in \overline{\Gcal}(\Theta): h(p_G,p_{G_0}) \geq L\epsilon_n} P_G^n (1-\phi_n) \nonumber\\ &\lesssim & \dfrac{\exp((1+C)n \lambda_{min}\epsilon_n^2)}{\frac{\Gamma(\gamma) (c_0\gamma\pi^{d/2})^{D}}{(2\Gamma(d/2+1))^{D}(2D)^{D-1}} 
		\left(\frac{\sqrt{\lambda_{min}}\epsilon_n}
		{2\bar{\theta}}\right)^{2(D-1) +dD} }\exp(-KnL^2\epsilon_n^2) +o(1).
\end{eqnarray}
The final step follows from simple computation similar to that of the Proof of Theorem 2.1 in~\cite{Ghosal-Ghosh-vanderVaart-00} and using the fact that $\dfrac{\bar{\theta}^{d}}{\epsilon^{d+2}_n} \log\left(\dfrac{\bar{\theta}}{\epsilon_n}\right) = o(n)$.
Combining equations~\eqref{eq:type1error_gaussian} and~\eqref{eq:type2error_gaussian} and using the condition $\dfrac{\bar{\theta}^{d}}{\epsilon^{d+2}_n} \log\left(\dfrac{\bar{\theta}}{\epsilon_n}\right) = o(n)$, it follows that for $L$ large enough 
\begin{eqnarray}
    \Pi_n(G \in \overline{\Gcal}(\Theta): h(p_G,p_{G_0}) \geq L\epsilon_n| X_1,\ldots, X_n) \overset{P_{G_0}}{\rightarrow} 0.
\end{eqnarray}

\paragraph{Step 2:} For some sufficiently large $L$ with $\epsilon_n=L(\log n)n^{-1/(d+2)}$ satisfies $\dfrac{\bar{\theta}^{d}}{\epsilon^{d+2}_n} \log\left(\dfrac{\bar{\theta}}{\epsilon_n}\right) = o(n)$. Therefore we get, from the result in Step 1of this proof
 \begin{eqnarray*}
    \Pi_n \biggr( G \in \overline{\Gcal}(\Theta): h(p_G,p_{G_0}) \geq  \dfrac{L(\log n)}{n^{1/(d+2)}} \; \bigg| \; X_{1:n} \biggr) \overset{P_{G_0}^{n}}{\rightarrow} 0.
\end{eqnarray*}
Now, from Theorem~\ref{theorem: gaussian_excess}, we have
 \begin{eqnarray}
    &\Pi_n \biggr(G \in \overline{\Gcal}(\Theta): W_{\Phi}(G,G_0) \geq  f_1(n,d) \; \bigg| \; X_{1:n}\biggr) \overset{P_{G_0}^{n}}{\rightarrow} 0, \nonumber
\end{eqnarray}
where $f_1(n,d):=(\log(n)/(d+2)- \log(\log n))^{-1/8}$.

\subsection{Proof of Corollary~\ref{corollary:gaussian_excess}}
\label{proof:corollary}

Let $G_0=\sum_{i=1}^{k_0} p^0_i\delta_{\theta^0_i}$, $G=\sum_{j=1}^k p_i\delta_{\theta_i}$. 
Suppose $\vec{q}=(q_{ij})_{1 \leq i \leq k_0,1\ \leq j \leq k} \in [0,1]^{k_0 \times k}$ is a coupling between $\vec{p_0}=(p^0_{1}, \ldots, p^0_{k_0})$ and
$\vec{p}=(p_{1}, \ldots, p_{k})$, with $\mathcal{Q}(\vec{p},\vec{p'})$ represents the space of all such couplings of $\vec{p_0}$ and $\vec{p}$. Using the proof technique similar to Lemma~\ref{lemma: orlicz_bound}, we get
\begin{align}
&\sum q_{ij}\exp ((\|\theta_i^0-\theta_j \|/k)^{\beta}) \nonumber \\ & \geq \sum q_{ij} \mathbbm{1}_{ \{ \|\theta_i^0-\theta_j \| \geq \eta \} }\exp ((\eta/k)^{\beta}) \nonumber \\
& \geq \sum p_{j} \mathbbm{1}_{ \{ \|\theta_i^0-\theta_j \| \geq \eta \text{ for all } i \} }\exp ((\eta/k)^{\beta}),\nonumber
\end{align}
for all $1 < \beta < 16/15$.

We denote $K= \inf\{\lambda \geq 0:  \sum p_{j} \mathbbm{1}_{ \{ \|\theta_i^0-\theta_j \| \geq \eta \text{ for all } i \}}\exp ((\eta/\lambda)^{\beta}) \leq 2 \}$. Then, we find that
\begin{align*}
& K \geq \eta \left( \log \left(\dfrac{1}{\sum p_{j} \mathbbm{1}_{ \{ \|\theta_i^0-\theta_j \| \geq \eta \text{ for all } i \}}}\right)\right)^{-1/\beta}, \quad \text{and} \\
& \sum_j p_j\mathbbm{1}_{ \{ \|\theta_j-\theta^0_i\|>\eta \text{ for all } i \} } \leq 2 \exp\left(\dfrac{-\eta}{W_{\Phi}(G,G_0)}\right). \nonumber
\end{align*}
Putting these results together with Theorem~\ref{theorem:posterior_gaussian_excess} leads to 
\begin{align*}
    \Pi_n \biggr(G \in \excessmass_{\eta} \left(\Theta, 2\exp\left( -\left(\dfrac{\eta\log(n)^{1/8}}{(d+2)}\right)^{\beta}\right)\right) \; \bigg| \;  X_{1:n}\biggr) \overset{P_{G_0}}{\rightarrow} 0
\end{align*}
in $P_{G_0}^n$ probability. Since this result holds for all $1<\beta< 16/15$, we obtain the conclusion.

\section{Proofs for Lemmas}
\label{sec:proof_lemma}
We now present the proofs for all lemmas in Section~\ref{section:gaussian_excess_mass}.
\subsection{Proof of Lemma~\ref{lemma: orlicz_metric}}
\label{subsec:proof:lemma: orlicz_metric}
We need to show the following properties of Orlicz-Wasserstein:
\begin{enumerate}
    \item[(i)] $W_{\Phi}(\nu_1,\nu_2)=W_{\Phi}(\nu_2,\nu_1)$ for any probability measures $\nu_1,\nu_2$ on $(\mathbb{R}^d,\|\cdot\|)$.
    \item[(ii)] $W_{\Phi}(\mu,\mu)=0$ for any probability measure $\mu$ on $(\mathbb{R}^d,\|\cdot\|)$.
    \item[(iii)] $W_{\Phi}(\nu_1,\nu_2) \leq W_{\Phi}(\nu_1,\nu_3) +W_{\Phi}(\nu_3,\nu_2) $ for any probability measures $\nu_1,\nu_2,\nu_3$ on $(\mathbb{R}^d,\|\cdot\|)$.
\end{enumerate}
(i) follows easily from the fact $\|x-y\|= $ is symmetric with respect to $x,y \in \mathbb{R}^{d}$ .

For  (ii) consider the coupling, $\nu(x,y)=\mu(x) \mathbbm{1}_{x=y}$, then it is clear to see that for any $k>0$, $\int _{\mathbb{R}^d \times \mathbb{R}^d}\Phi (\|x-y\|/k)\,d\nu(x,y)=1$ and therefore $W_{\Phi}(\mu,\mu)=0$.

For part (iii), assume that $W_{\Phi}(\nu_1,\nu_3)= k_1, W_{\Phi}(\nu_3,\nu_2)=k_2$. Then, it is enough to show that there exists a coupling $\nu$ of $\nu_1$ and $\nu_2$ such that $\int _{\mathbb{R}^d \times \mathbb{R}^d}\Phi (\|x-y\|/(k_1 + k_2))\, \mathrm{d}\nu(x,y)\leq 1$.

By results from~\cite{Villani-03,Villani-09}, there exists a coupling $\mu_1$ of $\nu_1$ and $\nu_3$ and a coupling $\mu_2$ of $\nu_2$ and $\nu_3$ such that,
\begin{eqnarray}
\int _{\mathbb{R}^d \times \mathbb{R}^d}\Phi (\|x-z\|/k_1)\,d\mu_1(x,z) &\leq & 1 \nonumber\\
\int _{\mathbb{R}^d \times \mathbb{R}^d}\Phi (\|z-y\|/k_2)\,d\mu_2(y,z) &\leq & 1.
\end{eqnarray}
Then, by a result in probability theory there exists a probability measure $\mu$ on $\mathbb{R}^d \times \mathbb{R}^d \times \mathbb{R}^d$ such that 
\begin{eqnarray}
\int_{x \in \mathbb{R}^d}  \mu(\mathrm{d}x,y,z) & = & \mu_2(y,z) \nonumber \\
\int_{x \in \mathbb{R}^d}  \mu(x,\mathrm{d}y,z) & = & \mu_1(x,z)
\end{eqnarray}
Define $\nu(x,y) :=\int_{z \in \mathbb{R}^d}  \mu(x,y,\mathrm{d}z) $. Then, we obtain that
\begin{eqnarray}
& &\int _{\mathbb{R}^d \times \mathbb{R}^d}\Phi (\|x-y\|/(k_1 + k_2))\, \mathrm{d}\nu(x,y) \nonumber \\ & = & \int _{\mathbb{R}^d \times \mathbb{R}^d \times\mathbb{R}^d}\Phi (\|x-y\|/(k_1 + k_2))\, \mathrm{d}\mu(x,y,z) \nonumber \\
& \leq & \int _{\mathbb{R}^d \times \mathbb{R}^d \times\mathbb{R}^d}\Phi ((\|x-z\| + \|y-z\|)/(k_1 + k_2))\, \mathrm{d}\mu(x,y,z) \nonumber\\
& \leq & \int _{\mathbb{R}^d \times \mathbb{R}^d \times\mathbb{R}^d}\Phi \left(\dfrac{k_1}{k_1+k_2}\dfrac{\|x-z\|}{k_1}+ \dfrac{k_2}{k_1+k_2}\dfrac{\|y-z\|}{k_2} \right)\, \mathrm{d}\mu(x,y,z) \nonumber\\& \leq & \dfrac{k_1}{k_1+k_2} \int _{\mathbb{R}^d \times \mathbb{R}^d }\Phi \left(\dfrac{\|x-z\|}{k_1}\right)\, \mathrm{d}\mu_1(x,z) \nonumber\\
& + & \dfrac{k_2}{k_1+k_2} \int _{\mathbb{R}^d \times \mathbb{R}^d }\Phi \left(\dfrac{\|y-z\|}{k_2}\right)\, \mathrm{d}\mu_2(y,z) \leq 1. \nonumber
\end{eqnarray}
The first inequality follows from the triangle inequality property of $\| \cdot\|$, while the last inequality follows from the convexity of $\Phi$. 

\subsection{Proof of Lemma~\ref{lemma: orlicz properties}}
\label{subsec:proof:lemma: orlicz properties}

Fix a coupling $\nu$ of $\nu_1$ and $\nu_2$. Consider $\lambda$ satisfying
\begin{eqnarray}
    & & \int _{\mathbb{R}^d \times \mathbb{R}^d}\Phi (\|x-y\|/\lambda)\,d\nu(x,y) < 
    \infty, \nonumber \\
    & & \int_{\mathbb{R}^d \times \mathbb{R}^d}\Psi (\|x-y\|/\lambda)\,d\nu(x,y) < \infty, \nonumber \\  
   & &\int _{\mathbb{R}^d \times \mathbb{R}^d}\Phi (\|x-y\|/\lambda)\,d\nu(x,y) \leq \int _{\mathbb{R}^d \times \mathbb{R}^d}\Psi (\|x-y\|/\lambda)\,d\nu(x,y), \nonumber
\end{eqnarray}
and thus, we find that
\begin{eqnarray}
    \biggr\{\lambda: &\int _{\mathbb{R}^d \times \mathbb{R}^d}\Psi (\|x-y\|/\lambda)\,d\nu(x,y) \leq 1 \biggr\} \\ &\subset \biggr\{\lambda: \int _{\mathbb{R}^d \times \mathbb{R}^d}\Phi (\|x-y\|/\lambda)\,d\nu(x,y) \leq 1 \biggr\} \nonumber.
\end{eqnarray}

As a consequence, we obtain the conclusion of Lemma~\ref{lemma: orlicz properties} since infimum of a set is smaller than the infimum of its subset.

\subsection{Proof of Lemma~\ref{lemma:convolution_bound_orlicz_Wasserstein}}
\label{subsec:proof:lemma:convolution_bound_orlicz_Wasserstein}

Consider $X \sim \nu_1$ and $Y \sim  Z_{\delta,d}$. Let $K$ be such that 
$$
\int_{\mathbb{R}} \exp((7/32)|y_i/K|^{\alpha}  - (7/16)|y_i/\delta|^ {4/3}) \mathrm{d}y_i < \infty.
$$
Then, we find that
\begin{eqnarray}
          & &\inf_{\mu}\biggr\{\int _{\mathbb{R}^d \times \mathbb{R}^d}\Phi (\|x-y\|/\lambda)\,d\mu(x,y): \mu \in \mathcal{Q}(\nu_1,\nu_2)\biggr\} \nonumber \\ &\leq & \left(\dfrac{1}{\delta}\right)^{d}\int_{\mathbb{R}^d} \exp((7/32)\|y\|^\alpha/\lambda^{\alpha}) \prod_{i=1}^d K_1(y_i/\delta) \prod_{i=1}^d \mathrm{d}y_i -1 \nonumber\\
         & \leq &\prod_{i=1}^d \left(\dfrac{1}{\delta}\right)\int_{\mathbb{R}} \exp((7/32)|y_i|^\alpha/\lambda^{\alpha})  K_1(y_i/\delta) \mathrm{d}y_i -1\nonumber \\
         & = & \prod_{i=1}^d \left(\dfrac{1}{\delta}\right)\int_{\mathbb{R}} \phi(y_i)^2 \exp((7/32)|y_i/\lambda|^{\alpha}   - (7/16)|y_i/\delta|^ {4/3}) \mathrm{d}y_i -1 \nonumber, 
\end{eqnarray}
where $\phi(\cdot)$ is the function in Lemma~\ref{lemma: fourier transform}. The second inequality follows from the fact that $\|x\|_{p} \leq \|x\|_{q}$ when $p \geq q$, where $\|\cdot\|_{p}$ is the $L_p$ norm. The final equality follows from Lemma~\ref{lemma: fourier transform}. Now, as $|\phi(x)| \leq C_{\phi}$ for some constant $C_{\phi} < \infty$, we have following the result in Lemma~\ref{lemma: orlicz properties},
\begin{eqnarray}
W_{\Phi}(\nu_1,\nu_2) \leq  C_{\alpha}\delta \nonumber
\end{eqnarray}
where 
\begin{eqnarray}
 C_{\alpha}=\inf\biggr\{ k>0: \int_{\mathbb{R}}\exp(|y/k|^{\alpha}   -|y|^ {4/3}) \mathrm{d}y -1 \leq \dfrac{1}{C_{\phi}^2}\biggr\} \nonumber.
\end{eqnarray}
Note that, $ C_{\alpha}$ as defined above exists because $\alpha \leq 4/3$. As a consequence, we obtain the conclusion of the lemma.
\subsection{Proof of Lemma~\ref{lemma: orlicz_wasserstein difference}}
\label{subsec:proof:lemma: orlicz_wasserstein difference}

Consider a coupling, $\nu$ between $\nu_1$ and $\nu_2$ that keeps fixed all the mass shared between $\nu_1$ and $\nu_2$, and redistributes the remaining mass independently, i.e.,
\begin{eqnarray}
\nu(x,y)= (\nu_1(x) \bigwedge \nu_2(y) ) \mathbbm{1}_{x=y} + \dfrac{1}{(\nu_1-\nu_2)_{+}(\mathbb{R}^d)}(\nu_1(x)-\nu_2(x))_{+} (\nu_2(y)-\nu_1(y))_{+}
\end{eqnarray}
Let $k_0$ be defined as 
\begin{eqnarray}
k_0:= \inf \{k \in \mathbb{R}^{+}: \int _{\mathbb{R}^d }\Phi (\|x\|/k)\,d|\nu_1(x)- \nu_2(x)|\leq 1\}.
\end{eqnarray}
Then, using $\nu$ as defined in the above display we get
\begin{eqnarray}
& &\int _{\mathbb{R}^d \times \mathbb{R}^d}\Phi (\|x-y\|/2k_0)\,d\nu(x,y)  \\
&=&\int _{\mathbb{R}^d \times \mathbb{R}^d}\Phi (\|x-y\|/2k_0)  \cdot \dfrac{1}{(\nu_1-\nu_2)_{+}(\mathbb{R}^d)}(\nu_1(x)-\nu_2(x))_{+} (\nu_2(y)-\nu_1(y))_{+}\nonumber \\ 
& \leq &  \int _{\mathbb{R}^d \times \mathbb{R}^d}\Phi (\|x\|/k_0) (\nu_1(x)-\nu_2(x))_{+} \leq 1 \nonumber
\end{eqnarray}

Therefore,
\begin{eqnarray}
& W_{\Phi}(\nu_1,\nu_2) \leq 2 \inf \{k \in \mathbb{R}^{+}: \int _{\mathbb{R}^d }\Phi (\|x\|/k)\,d|\nu_1(x)- \nu_2(x)|\leq 1\}. \nonumber
\end{eqnarray}
As a consequence, we reach the conclusion of the lemma.

\begin{lemma}
\label{lemma: fourier transform}
Let $f(x)=\exp(-x^4)$, and $\tilde{f}(t)=(1/2\pi)\int_{-\infty}^{\infty} \exp(-itx)f(x)\mathrm{d}x$. Then, 
\begin{eqnarray}
     |\tilde{f}(t)| \leq \phi(t)\exp(-7/32 |t|^{4/3}),
\end{eqnarray}
where $\phi(t)$ is an absolutely bounded real-valued function.
\end{lemma}
\begin{proof}
Consider a rectangle on the complex plane, with vertices at $R,-R, R+i\zeta, -R+ i\zeta$ respectively.
Following Goursat's Theorem~\cite{Stein-Shakarchi-complex} for integration along rectangular contours on the complex plane, the contour integral along a closed rectangle is $0$.

Therefore, 

\begin{eqnarray}
 & &\int_{-R}^R \exp(-itx)f(x)\mathrm{d}x + \int_{R}^{R+i \zeta} \exp(-itx)f(x) \mathrm{d}x + \int_{-R+ i\zeta}^{-R} \exp(-itx)f(x) \mathrm{d}x \nonumber \\ & & +\int_{R+i\zeta}^{-R+i \zeta} \exp(-itx)f(x)\mathrm{d}x=0.\nonumber
\end{eqnarray}
Now, 
\begin{eqnarray}
|\int_{R}^{R+i \zeta} \exp(-itx)f(x) \mathrm{d}x| =|\int_{0}^{\zeta} \exp(itR -tx)f(R+ix)i \mathrm{d}x| \leq C \exp(-R^4) \to 0, \nonumber
\end{eqnarray}
as $R \to \infty$.
Similarly,
\begin{eqnarray}
|\int_{-R+ i\zeta}^{-R} \exp(-itx)f(x) \mathrm{d}x| \to 0, \nonumber
\end{eqnarray}
as $R \to \infty$.

Therefore,
\begin{eqnarray}
\lim_{R \to \infty}\int_{-R+i\zeta}^{R+i \zeta} \exp(-itx)f(x)\mathrm{d}x = \lim_{R \to \infty}\int_{-R}^R \exp(-itx)f(x)\mathrm{d}x = 2\pi\tilde{f}(t).\nonumber
\end{eqnarray}

Now,
\begin{eqnarray}
\lim_{R \to \infty}\int_{-R}^R \exp(-itx)f(x)\mathrm{d}x = 2\pi\tilde{f}(t) &= & \lim_{R \to \infty}\int_{-R+i\zeta}^{R+i \zeta} \exp(-itx)f(x)\mathrm{d}x \nonumber \\ &=& \lim_{R \to \infty}\int_{-R}^{R} \exp(it(x+i\zeta))f(x +i \zeta)\mathrm{d}x.\nonumber \\
&=& \lim_{R \to \infty}\int_{-R}^{R} \exp(-itx-t\zeta))\exp(-(x+i \zeta)^4)\mathrm{d}x \nonumber .
\end{eqnarray}
Expanding the above expression,
\begin{eqnarray}
&\tilde{f}(t)= (1/2\pi) \lim_{R \to \infty} \int_{-R}^R \exp( -itx-4ix^3\zeta   +4ix \zeta^3 -t\zeta -(x^2 -3 \zeta^2)^2+8\zeta^4)\mathrm{d}x \nonumber.
\end{eqnarray}
Substituting $\zeta=\dfrac{1}{4}\sign(t)|t|^{1/3}$ in the above equationa,
\begin{eqnarray}
|\tilde{f}(t)| \leq (1/2\pi)\exp(-(7/32)|t|^{4/3} )
  \cdot \int_{-\infty}^{\infty} \exp(-(x^2 -(1/3)|t|^{1/2})^2)\mathrm{d}x.
\end{eqnarray}
The proof is complete when we note that $ \phi(t) := (1/2\pi)\int_{-\infty}^{\infty} \exp(-(x^2 -(1/3)|t|^{1/2})^2)\mathrm{d}x$ is an absolutely bounded function.
\end{proof}

\begin{lemma}
\label{lemma: fourier transform_2}
Let $k(t)= c\tilde{f}(t)^2$, where  $\tilde{f}(t)=(1/2\pi)\int_{-\infty}^{\infty}\exp(-itx)\exp(-x^4)\mathrm{d}x$ and $c$ is a constant of proportionality so that $\int_{-\infty}^{\infty} k(t) \mathrm{d}t=1$. Then,
\begin{eqnarray}
|\int_{-\infty}^{\infty}\exp(itx)k(t)\mathrm{d}t | \lesssim  \exp(-(x/2)^4)
\end{eqnarray}
\end{lemma}
\begin{proof}
Define $f(x)=\exp(-x^4)$. Then, by a version of the Fourier inversion theorem,
\begin{eqnarray}
\int_{-\infty}^{\infty}\exp(itx)k(t)\mathrm{d}t =f \ast f(x) \nonumber,
\end{eqnarray}
where $\ast$ is the convolution operator. Since convolution of even functions is even, it is enough to show the result for $x>0$. Then,
\begin{eqnarray}
f \ast f(x) & =\int_{-\infty}^{\infty} \exp(-y^4) \exp(-(y-x)^4) \mathrm{d}y \nonumber\\
& = \int_{x/2}^{\infty} \exp(-y^4) \exp(-(y-x)^4)\mathrm{d}y  + \int_{-\infty}^{x/2} \exp(-y^4) \exp(-(y-x)^4)\mathrm{d}y \nonumber \\
& \leq  \exp(-(x/2)^4)\int_{x/2}^{\infty}  \exp(-(y-x)^4)\mathrm{d}y + \exp(-(x/2)^4)\int_{-\infty}^{x/2} \exp(-y^4) \mathrm{d}y\nonumber \\
& \leq  2\exp(-(x/2)^4)\int_{-\infty}^{\infty}  \exp(-y^4)\mathrm{d}y.
\end{eqnarray}
The result holds with $C=2\int_{-\infty}^{\infty}  \exp(-y^4)\mathrm{d}y$ since $\int_{-\infty}^{\infty}  \exp(-y^4)\mathrm{d}y < \infty$.
\end{proof}
\subsection{Proof of Lemma~\ref{lemma: orlicz_bound}}
\label{subsec:proof:lemma: orlicz_bound}
Suppose $\vec{q}=(q_{ij})_{1 \leq i \leq k_0,1\ \leq j \leq k} \in [0,1]^{k_0 \times k}$ is a coupling between $\vec{p_0}=(p^0_{1}, \ldots, p^0_{k_0})$ and
$\vec{p}=(p_{1}, \ldots, p_{k})$, with $\mathcal{Q}(\vec{p},\vec{p'})$ representing the space of all such couplings of $\vec{p}$ and $\vec{p'}$. Then, for fixed $k$ we have
\begin{align*}
\sum q_{ij}\Phi (\|\theta_i^0-\theta_j \|/k) & \geq \sum q_{ij} \mathbbm{1}_{ \{ \|\theta_i^0-\theta_j \|\geq \eta \} }\Phi (\eta/k) \\
&\geq \sum p_{j} \mathbbm{1}_{ \{ \|\theta_i^0-\theta_j \| \geq \eta \text{ for all } i \} }\Phi (\eta/k).
\end{align*}
Let $K= \inf\{k \geq 0:  \sum p_{j} \mathbbm{1}_{ \{ \|\theta_i^0-\theta_j \| \geq \eta \text{ for all } i \}} \Phi (\eta/k) \leq 1 \}$. Then,
\begin{eqnarray}
\label{eq: K and eta}
K \geq \eta \left( \Phi^{-1}\left(\dfrac{1}{\sum p_{j} \mathbbm{1}_{ \{ \|\theta_i^0-\theta_j \| \geq \eta \text{ for all } i \} }}\right)\right)^{-1},
\end{eqnarray}
where $\Phi^{-1}$ is the inverse function of the function $\Phi$. Note that, this function exists and is concave as $\Phi$ is monotonic increasing and convex.
Moreover,  by Lemma~\ref{lemma: orlicz properties}(i), we would have that $W_{\Phi}(G,G_0) \geq K$, where,

\begin{eqnarray}
\label{eq: W and eta}
& W_{\Phi}(G,G_0):= \inf_{q \in \mathcal{Q}(\vec{p},\vec{p'})} \{ \inf\{k \geq 0: \sum q_{ij}\Phi (\|\theta_i^0-\theta_j \|/k) \leq 1 \} \}  \nonumber
\end{eqnarray}
Combining the results from equations~\eqref{eq: K and eta} and~\eqref{eq: W and eta} we obtain the conclusion of the lemma.
\subsection{Prior mass on Wasserstein ball}
\label{subsection:proof_lemma_prior_mass_MFM}

\begin{lemma}
\label{lemma:Prior_mass_DPMM}
Let $G \sim  DP(\gamma,H_n)$. Fix $r\geq 1$. Assume $G_{0} \in \mathcal{M}(\Theta)$, where $\Theta= [-\bar{\theta},\bar{\theta}]^{d}$. If $H_n$ admits condition (P.1), then the following holds

\begin{eqnarray}
& & \Pi \left( W_r^r(G,G_{0}) \leq (2^r +1)\epsilon^r \right) \geq \frac{\Gamma(\gamma) (c_0\gamma\pi^{d/2})^{D}}{(2\Gamma(d/2+1))^{D}(2D)^{D-1}} 
		\left(\frac{\epsilon}
		{2\bar{\theta}}\right)^{r(D-1) +dD} \nonumber
\end{eqnarray}
for all $\epsilon$ sufficiently small so that $D(\epsilon, \Theta, \|.\|)>\gamma$. 

Here, $D = D(\epsilon, \Theta, \|.\|)$ stands for the maximal $\epsilon$-packing number for $\Theta$ under $\|.\|$ norm, and $\Gamma(\cdot)$ is the gamma function.

\end{lemma}

\begin{proof}

From Lemma 5 in~\cite{Nguyen-13},
\begin{eqnarray}
  \Pi \left( W_r^r(G,G_{0}) \leq (2^r +1)\epsilon^r \right) 
\geq  \frac{\Gamma(\gamma) \gamma^{D}}{(2D)^{D-1}} 
		\left(\frac{\epsilon}
		{\text{Diam}(\Theta)}\right)^{r(D-1)} \sup_S \prod_{i=1}^{D} H_n(S_i) \nonumber , 
\end{eqnarray}
where, $S := (S_1,...,S_{D})$ denotes the ${D}$ disjoint  $\epsilon/2$-balls that form a maximal $\epsilon$-packing of $\Theta$. The supremum is taken over all such packings.

Now,  $H_n(A) \geq \left(\dfrac{c_0}{\mu(\Theta)}\right)\mu(A)$. Moreover, $\prod_{i=1}^{D} \mu(S_i) \geq \left(\dfrac{(\sqrt{\pi}\epsilon)^d}{2\Gamma(d/2+1)}\right)^{D}$. Using this, we arrive at the result.

\end{proof}

\subsection{Metric entropy with Hellinger distance}

\begin{lemma}
\label{lemma:Hellinger_metric_entropy}
Let $G_0$ be a discrete mixing measure with all its atoms in $\Theta=[-\tilde{\theta},\tilde{\theta}]^d \subset \mathbb{R}^d$. Let $\mathscr{P}_{\overline{\Gcal}(\Theta)}:=\{p_G: G \in \overline{\Gcal}(\Theta)\}$. Then, if the kernel $f$ is multivariate Gaussian with covariance matrix $\Sigma$,
\begin{eqnarray}
\label{eq:hellinger_entropy_gaussian}
\log D(\epsilon/2, \{p_G \in \mathscr{P}_{\overline{\Gcal}(\Theta)}: \epsilon< h(p_G,p_{G_0}) \leq 2\epsilon \}, h) \leq c_1 \left(\dfrac{\tilde{\theta}}{\sqrt{\lambda_{min}}\epsilon}\right)^d \log\left(e + \dfrac{32e \tilde{\theta}^2}{\lambda_{min}\epsilon^2}\right)
\end{eqnarray}
for some universal constant $c_1$.

\end{lemma}
\begin{proof}

Let $N(\epsilon,\mathscr{P},d)$ denote the $\epsilon$-covering number of the space $\mathscr{P}$ relative to the metric $d$. It is related to the packing number by the following identity:
\begin{eqnarray}
     \label{eq:covering_packing}
     N(\epsilon, \mathscr{P}, h) \leq D(\epsilon, \mathscr{P}, d) \leq N(\epsilon/2, \mathscr{P}, h).
\end{eqnarray}
 Using the result in Example 1 of~\cite{Nguyen-13}, when $f_{\Sigma}(\cdot|\theta)\sim \mathcal{N}_d(\theta,\Sigma)$,
    \begin{eqnarray}
    \label{eq:hellinger_gaussian_bound}
    h^2(f_{\Sigma}(\cdot|\theta_i), f_{\Sigma}(\cdot|\theta'_j)) = 1- \exp\left(-\dfrac{1}{8}\|\theta_i-\theta'_j\|_{\Sigma^{-1}}^2\right) \leq \dfrac{\|\theta_i-\theta'_j\|^2}{8\lambda_{min}},
    \end{eqnarray} 
    where $\|z\|_{\Sigma^{-1}}:=\sqrt{z'\Sigma^{-1}z}$.

Let $G_0=\sum_{i=1}^{k_0} p_{i}^0\delta_{\theta_i^0} $ and $G=\sum_{j=1}^{k'} p'_{j}\delta_{\theta'_j}$ be mixing measures in $\overline{\Gcal}(\Theta)$, with  $k_0,k' \in [1,\infty]$. Let $\vec{q}=(q_{ij})_{1 \leq i \leq k_0 ,1\ \leq j \leq k'} \in [0,1]^{k_0 \times k'}$ denote a coupling of $\vec{p^0}$ and $\vec{p'}$.
    
Using Lemma 2 of~\cite{Nguyen-13} with $\phi(x)= \dfrac{1}{2}(\sqrt{x}-1)^2$, gives us:
\begin{eqnarray}
h^2(p_G,p_{G_0}) \leq \inf_{\vec{q}\in Q(\vec{p_0},\vec{p'})} \sum_{i,j} q_{ij}\dfrac{\|\theta_i-\theta'_j\|^2}{8\lambda_{min}} = \dfrac{W_2(G,G_0)^2}{8\lambda_{min}},
\end{eqnarray}
where $Q(\vec{p_0},\vec{p'})$ is the set of all couplings of $\vec{p_0}$ and $\vec{p'}$.
Therefore, it immediately follows that:
\begin{eqnarray}
& & \log D(\epsilon/2, \{p_G \in \mathscr{P}_{\overline{\Gcal}(\Theta)}: \epsilon< h(p_G,p_{G_0}) \leq 2\epsilon \}, h) \nonumber \\ &\leq & \log D(\sqrt{2 \lambda_{min}} \epsilon, \{G: G \in \overline{\Gcal}(\Theta)\}, W_2)\leq  N\left (\sqrt{\dfrac{ \lambda_{min}}{8}} \epsilon, \Theta, \|\cdot\|\right) \log\left(e + \dfrac{32e \tilde{\theta}^2}{\lambda_{min}\epsilon^2}\right) \nonumber.
\end{eqnarray}
The last inequality follows by applying Eq.~\eqref{eq:covering_packing} followed by Lemma 4 part (b) of~\cite{Nguyen-13}. The result then follows immediately.

\end{proof}

\subsection{Computation of $M$ corresponding to KL ball}
\begin{lemma}
\label{lemma:M_compute}
Let $G$ be a discrete mixing measure with all its atoms in $\left[-\tilde{\theta},\tilde{\theta}\right]^d$ for some $\tilde{\theta} >0$. Furthermore, assume the atoms of $G_0$ lie in $\left[-\bar{\theta},\bar{\theta} \right]^d$ where $\bar{\theta} > 0$ is given. Then, the following holds if the kernel $f$ is multivariate Gaussian,
\begin{eqnarray}
\bigintsss \dfrac{(p_{G_0}(x))^2}{p_G(x)}\mu(\mathrm{d}x) \leq \exp(d\lambda_{min}^{-1} (  5\bar{\theta}^2 + 4\tilde{\theta}^{2})).
\end{eqnarray}
Here $\mu$ is the Lebesgue measure on $\mathbb{R}^d$.
\end{lemma}
\begin{proof}

For the multivariate Gaussian kernel with covariance matrix $\Sigma$, similar to the multivariate Laplace case, using lemma 2 of~\cite{Nguyen-13} with $\phi(x)=\dfrac{1}{x}$, gives us:
    \begin{eqnarray}
    \label{eq:M_bound_gaussian}
    \bigintsss \dfrac{(p_{G_0}(x))^2}{p_G(x)}\mu(\mathrm{d}x) \leq \inf_{\vec{q}\in Q(\vec{p_0},\vec{p'})} \sum_{i,j} q_{ij}\bigintsss \dfrac{(f_{\Sigma}(x|\theta_i^0))^2}{f_{\Sigma}(x|\theta_j')}\mu(\mathrm{d}x),
    \end{eqnarray}
    where $Q(\vec{p_0},\vec{p'})$ is the set of all couplings of $\vec{p_0}$ and $\vec{p'}$, and $f_{\Sigma}(\cdot|\theta)$ is the multivariate Gaussian kernel with covariance parameter $\Sigma$ and mean parameter $\theta$. 
    
\begin{eqnarray}
\label{eq:gaussian_exact_bound_M}
 \bigintsss \dfrac{(f_{\Sigma}(x|\theta_i^0))^2}{f_{\Sigma}(x|\theta_j')}\mu(\mathrm{d}x)&=&\bigintsss f_{\Sigma}(x|\theta_i^0) \exp\left(\dfrac{-\|x-\theta_i^0\|^2_{\Sigma^{-1}} + \|x-\theta_j'\|^2_{\Sigma^{-1}}}{2}\right)\mu(\mathrm{d}x) \\
 &=&\bigintsss f_{\Sigma}(x|\theta_i^0) \exp\left(\dfrac{-\|\theta_j'-\theta_i^0\|^2_{\Sigma^{-1}}}{2} + \langle x-\theta_j', \Sigma^{-1}\theta_j'-\theta_i^0\rangle \right)\mu(\mathrm{d}x)
 \nonumber,
\end{eqnarray}
 where the second equality follows by simple calculation using $x-\theta_i^0=(x-\theta_j') +(\theta_j'-\theta_i^0)$.
 
If $M_{\Sigma}(t|\theta)$ is the moment generating function of the Gaussian distribution with mean $\theta$ and covariance $\Sigma$, then 
\begin{eqnarray}
M_{\Sigma}(t|\theta)= \exp(\langle \theta,t \rangle + \dfrac{1}{2}\langle t, \Sigma t\rangle)\nonumber.
\end{eqnarray}
Using this result , we can rewrite Eq.~\eqref{eq:gaussian_exact_bound_M} as 
\begin{eqnarray}
 \bigintsss \dfrac{(f_{\Sigma}(x|\theta_i^0))^2}{f_{\Sigma}(x|\theta_j')}\mu(\mathrm{d}x)&=&\exp(\langle \theta_j'-\theta_i^0, \Sigma^{-1}\theta_i^0+\theta_j'\rangle ) \leq \exp(2d \lambda_{min}^{-1} (\tilde{\theta} + \bar{\theta})^2 + d\lambda_{min}^{-1}\bar{\theta}^{2})
 \nonumber,
    \end{eqnarray}
The bound on $\int (p_{G_0}(x))^2/p_G(x)\mu(\mathrm{d}x) $ then follows immediately.
\end{proof}


\section{ Theoretical guarantee of Algorithm~\ref{algo:compute_orlicz}}
\label{ssection:Entropic_OT}

We show in this section that the output of Algorithm~\ref{algo:compute_orlicz} converges to the Entropy regularised version of the Orlicz-Wasserstein distance in equation~\eqref{eq:surrogate_orlicz-Wasserstein}.
\begin{proposition}
\label{prop: algo_proof}
Let $\hat{W}^{\lambda}_{\Phi}(\nu_1,\nu_2)$ be the output of Algorithm~\ref{algo:compute_orlicz} and $W^{\lambda}_{\Phi}(\nu_1,\nu_2)$ be as in equation~\eqref{eq:surrogate_orlicz-Wasserstein}. Then 
\begin{eqnarray}
    |\hat{W}^{\lambda}_{\Phi}(\nu_1,\nu_2)-W^{\lambda}_{\Phi}(\nu_1,\nu_2)|<\epsilon.
\end{eqnarray}
\end{proposition}

\begin{proof}
Here M is the cost matrix such that $M_{ij}=\|x_i-y_j\|$.

Note that $S(\Phi(M/W^{\lambda}_{\Phi}(\nu_1,\nu_2)),\lambda,r,c) <1$ and if $S(\Phi(M/\eta),\lambda,r,c) <1$, then $\eta<W^{\lambda}_{\Phi}(\nu_1,\nu_2))$.

If $\hat{x}_{upp}=max(M)/\Phi^{-1}(1)$, $\hat{x}_{low} = d(M,\lambda,r,c)/\Phi^{-1}(1+d(M,\lambda,r,c)-S(M,\lambda,r,c)) $ it is enough to show that \\ $fx_{upp}=S(\Phi(M/\hat{x}_{upp}),\lambda,r,c)<1$,$fx_{low}=S(\Phi(M/\hat{x}_{low}),\lambda,r,c)>1$, since it would imply $x_{upp}:=\hat{W}^{\lambda}_{\Phi}(\nu_1,\nu_2)< W^{\lambda}_{\Phi}(\nu_1,\nu_2)<x_{low}$ and therefore if $|x_{upp}-x_{low}|< \epsilon$, the result holds directly.

We need to show 
\begin{enumerate}
    \item[(i)] $S(\Phi(M/\hat{x}_{upp}),\lambda,r,c)<1$.
    \item [(ii)] $S(\Phi(M/\hat{x}_{low}),\lambda,r,c)>1$.
    \end{enumerate}

For (i), observe that 
\begin{eqnarray}
S(\Phi(M/\hat{x}_{upp}),\lambda,r,c)= \inf_{\nu \in \mathcal{Q}(\nu_1,\nu_2)}\int _{\mathbb{R}^d \times \mathbb{R}^d}\Phi (\|x-y\|/\hat{x}_{upp})\,d\nu(x,y) -(1/\lambda) (H(\nu)) \\
= \inf_{\nu \in \mathcal{Q}(\nu_1,\nu_2)}\int _{\mathbb{R}^d \times \mathbb{R}^d}\Phi (\Phi^{-1}(1)\|x-y\|/max(M))\,d\nu(x,y) -(1/\lambda) (H(\nu)) \leq 1
\end{eqnarray}
The last inequality holds by monotonicity of $\Phi$ combined with $\|x-y\|/max(M) <1$ with $\nu$-probability 1, and the fact that $H(\nu)>0$.

For (ii),note that for any $\nu \in \mathcal{Q}(\nu_1,\nu_2)$, it holds that 
\begin{eqnarray}
& &\int _{\mathbb{R}^d \times \mathbb{R}^d}\Phi (\|x-y\|/\eta)\,d\nu(x,y) -H(\nu)/\lambda \nonumber \\ 
& & \hspace{3 em}\geq  \Phi\left(\int _{\mathbb{R}^d \times \mathbb{R}^d} (\|x-y\|/\eta)\, d\nu(x,y)\right)  -(H(r)+H(c))/\lambda \\
& & \hspace{3 em} \geq \Phi((S(M,\lambda,r,c) + (H(r)+H(c))/2\lambda)/\eta))  -(H(r)+H(c))/\lambda. 
\end{eqnarray}
Both the inequalities hold by monotonicty and convexity of $\Phi$ combined with the fact that $\forall \nu \in \mathcal{Q}(\nu_1,\nu_2)$, it holds that $H(r)+H(c)\geq H(\nu)\geq (H(r)+H(c))/2$. 

Now $\Phi((S(M,\lambda,r,c) + (H(r)+H(c))/2\lambda)/\eta))  -(H(r)+H(c))/\lambda \geq 1$, for any  $\eta\leq \hat{x}_{upp}$, this completes the proof.
\end{proof}

\section{Estimation of number of components for mixing measures}
\label{section: number of components}
In this section, we consider how Orlicz-Wasserstein distances could be used to improved estimation of the number of components with Gaussian mixture models. Gaussian Mixture models have been used for the purpose of clustering both historically~\cite{GMMneurips2004} as well as in modern applications~\cite{GMM-clustering-19,GMM_clustering_purvasha_19,GMM-clustering-22}. From the Bayesian perspective, often used BNP priors for mixture models tend to overestimate the number of components drastically by producing multiple extraneous components around the "true" components~\cite{Miller-2014}. This makes it difficult to estimate the number of components, where it may of interest~\cite{MacEachern-Muller-98,Green-Richardson-clustering}.

Several recent works have explored the consistent estimation of the number of components with mixture models, both with in-processing~\cite{Manole-Group-sort-fuse} and post-processing~\cite{Guha-MTM-19} techniques. However, while~\cite{Manole-Group-sort-fuse} restricts attention to the overfit setting only,~\cite{Guha-MTM-19} requires the knowledge of explicit contraction rates of respective parameters in Wasserstein distances. As parameter convergence rates of Dirichlet Process Gaussian Mixture models are extremely slow~\cite{Nguyen-13}, this would also affect the estimation of the number of components negatively. The procedures in both the works~\cite{Guha-MTM-19,Manole-Group-sort-fuse} consist of two smaller steps, truncation of extraneous outlier atoms and merging of atoms which are close to the "true" atoms. The results in this work specifically, Theorem~\ref{theorem:posterior_gaussian_excess} provide a low threshold for truncating outlier atoms thereby eliminating outlier atoms more efficiently. Combined with an understanding of convergence behavior around the "true" atoms would allow efficient estimation of the number of components with Dirichlet Process Gaussian Mixture models.

\end{document}